\documentclass[a4paper,11pt,reqno]{amsart}
\usepackage[centering,includeheadfoot,margin=2.5 cm]{geometry}
\usepackage[percent]{overpic}
\usepackage{todonotes}
\usepackage{enumitem}
\usepackage{graphics,enumitem,epsfig,textcomp}
\usepackage{amsfonts,amsmath,amssymb,euscript,color,mathrsfs, bbm}
\usepackage[utf8]{inputenc}	
\usepackage{float} 

\usepackage[caption=false,subrefformat=parens,labelformat=parens]{subfig}
\usepackage{url}
\usepackage[symbol]{footmisc}
\usepackage{hyperref}
\usepackage{cases}
\usepackage{bigints}
\hypersetup{
    colorlinks=true,
    linkcolor=blue,
    filecolor=blue,      
    urlcolor=blue,
    citecolor=blue,
    }
\usepackage{comment}

\usepackage[square,numbers]{natbib}
\usepackage{amsmath,amsthm,amssymb,amsfonts,lineno}
\usepackage{esint}
\usepackage{braket, leftindex}

\numberwithin{equation}{section}
\newtheorem{definition}{Definition}
\newtheorem{theorem}{Theorem}

\newtheorem{proposition}{Proposition}
\newtheorem{remark}{Remark}

\def\N{\mathbb{N}}
\def\R{\mathbb{R}}
\def\C{\hbox{\rlap{\kern.24em\raise.1ex\hbox{\vrule height1.3ex width.9pt}}C}}

\def\P{\hbox{\rlap{I}\kern.16em P}}
\def\Q{\hbox{\rlap{\kern.24em\raise.1ex\hbox
      {\vrule height1.3ex width.9pt}}Q}}
\def\M{\hbox{\rlap{I}\kern.16em\rlap{I}M}}
\def\Z{\hbox{\rlap{Z}\kern.20em Z}}
\def\({\begin{eqnarray}}
\def\){\end{eqnarray}}
\def\[{\begin{eqnarray*}}
\def\]{\end{eqnarray*}}
\def\part#1#2{\frac{\partial #1}{\partial #2}}

\def\pmb#1{\setbox0=\hbox{$#1$}
  \kern-.025em\copy0\kern-\wd0
  \kern-.05em\copy0\kern-\wd0
  \kern-.025em\raise.0433em\box0 }
\def\bar{\overline}

\def\weakconv{\rightharpoonup}

\def\N{\mathbb{N}}
\def\R{\mathbb{R}}

\def\epsilon{\varepsilon}

\def\P{\mathbb{P}}
\def\Q{\mathbb{Q}}

\usepackage{cancel}
\title{PDE Models for Deep Neural Networks: Learning Theory, Calculus of Variations and Optimal Control}
\date{October 2024}

\begin{document}
\pagenumbering{gobble}
\maketitle
\pagenumbering{arabic}

\centerline{
     {\large Peter Markowich}\footnote{Mathematical and Computer Sciences and Engineering Division,
         King Abdullah University of Science and Technology,
         Thuwal 23955-6900, Kingdom of Saudi Arabia;
         {\it peter.markowich@kaust.edu.sa}, and
         Faculty of Mathematics, University of Vienna,
        Oskar-Morgenstern-Platz 1, 1090 Vienna;
         {\it peter.markowich@univie.ac.at}}\qquad
         {\large Simone Portaro}\footnote{Mathematical and Computer Sciences
            and Engineering Division,
         King Abdullah University of Science and Technology,
         Thuwal 23955-6900, Kingdom of Saudi Arabia;
         {\it simone.portaro@kaust.edu.sa}}
     }
\medskip \medskip

\section*{Abstract}
We propose a partial differential-integral equation (PDE) framework for deep neural networks (DNNs) and their associated learning problem by taking the continuum limits of both network width and depth. The proposed model captures the complex interactions among hidden nodes, overcoming limitations of traditional discrete and ordinary differential equation (ODE)-based models. We explore the well-posedness of the forward propagation problem, analyze the existence and properties of minimizers for the learning task, and provide a detailed examination of necessary and sufficient conditions for the existence of critical points.

Controllability and optimality conditions for the learning task with its associated PDE forward problem are established using variational calculus, the Pontryagin Maximum Principle, and the Hamilton-Jacobi-Bellman equation, framing the deep learning process as a PDE-constrained optimization problem. In this context, we prove the existence of viscosity solutions for the latter and we establish optimal feedback controls based on the value functional. This approach facilitates the development of new network architectures and numerical methods that improve upon traditional layer-by-layer gradient descent techniques by introducing forward-backward PDE discretization.

The paper provides a mathematical foundation for connecting neural networks, PDE theory, variational analysis, and optimal control, partly building on and extending the results of \cite{liu2020selection}, where the main focus was the analysis of the forward evolution. By integrating these fields, we offer a robust framework that enhances deep learning models' stability, efficiency, and interpretability.

\section{Introduction}
Deep learning enables computational models with multiple processing layers to learn data representations at various levels of abstraction. This approach has significantly advanced the state-of-the-art in fields like speech recognition, visual object recognition \cite{bengio2009learning, lecun2015deep} and extends to areas such as drug discovery and genomics \cite{vamathevan2019applications}. By employing the backpropagation algorithm, deep learning uncovers complex structures in large datasets, guiding the adjustment of internal parameters to refine the representation at each layer based on the previous one. Given its extensive use, establishing a robust mathematical framework to analyze Deep Neural Networks (DNNs) is essential.

DNNs excel in supervised learning, particularly in scenarios where the data-label relationship is highly nonlinear. Their multiple layers allow DNNs to capture complex patterns by transforming features through each layer, effectively filtering the information content. 
The term "depth" of a DNN refers to its total number of layers, including both the hidden and output layers. The term "width" of a DNN refers to the number of neurons or units within each layer. Thus, a network's depth indicates its hierarchical level of data processing, while its width indicates the complexity and capacity of each layer to represent features.

In the literature, discrete neural networks are predominant because they are simple to program and they have excellent approximation properties \cite{hornik1989multilayer}. Taking the depth continuum limit transforms the discrete network into a dynamical system, which facilitates the understanding of complex discrete structures.

Previous research on the dynamical systems approach to deep learning has concentrated on algorithm design and enhancing network architecture using ordinary differential equations (ODEs) to model residual neural networks \cite{chang2017multi, haber2018learning}. However, ODE models do not show the structure of hidden nodes in relation to network width. To address this gap, we propose a partial differential-integral equation (PDE) model for DNNs that is derived via continuum limits in both width and depth and accounts for multiple, different, weakly linearly independent initial data. Consequently, the learning problem can be view as a data-fitting approach and formulated as a PDE-constrained optimization problem. The scenario with a single learning datum, albeit limited for a comprehensive study of DNNs, has been extensively examined in \cite{liu2020selection} as an initial approach. In this work we go far beyond the previous study, analyzing the induced coupling effects of multiple learning data on the network dynamics. In real-world applications of DNNs, using multiple learning data instead of a single datum is crucial, as it allows the model to process bigger data sets and deal with different inputs, capturing more complex patterns and relationships. This leads to improved robustness, accuracy, and performance in tasks where the variability and complexity of real-world data cannot be adequately represented by a single data point as in previous works. In a mathematical framework, the difficulty of the controllability of the forward problem and the Hamilton-Jacobi-Bellman equation becomes much richer in the multi-data setting, which is one of the novelties of our approach. Note that a key advantage of our PDE model over the ODE model \cite{haber2017stable} is its ability to capture the intrinsic dynamics among hidden units. 

Additionally, by discretizing forward and backward PDE problems using numerical methods, we can develop network architectures distinct from those based on the empirical explicit Euler scheme, which is integral to the depth continuum process. The diverse tools available in numerical analysis for PDEs provide enhanced stability, efficiency, and speed compared to traditional layer-by-layer iteration techniques.

In many applications, it is practical to limit the learning parameters to bounded sets, transforming the minimization process into a control theory problem. This approach results in coupled forward-backward PDEs connected through optimal controls. Consequently, the deep learning problem can be studied within the framework of mathematical control theory \cite{evans2005introduction, zabczyk2020mathematical}, following the Pontryagin Maximum Principles as described in \cite{lewis2006maximum, pontryagin2018mathematical} or the Dynamic Programming Principle \cite{evans2005introduction} through the Hamilton-Jacobi-Bellman equation. While the former only provides a necessary condition for optimality the latter also gives (in a sense) a 'sufficient' condition albeit at the expense of much greater complexity. The intersection of deep learning, dynamical systems, and optimal control has garnered growing interest \cite{chang2017multi, haber2017stable, han2019mean, li2018maximum, li2017deep, liu2020selection, sonoda2017double}. A notable advantage of this approach is its explicit consideration of the compositional structure in the time evolution of dynamical systems, paving the way for novel algorithms and network architectures. Numerical methods from control theory and mean field games can then replace traditional techniques like adapted gradient descent used in neural networks. Often, constraining the parameter space proves more effective than using regularization methods such as Tikhonov regularization.

We remark  that (one of) the main tasks of AI is to provide reasonably accurate models for data classification and functional data approximation. In the framework of our deep residual network approach this is done by first determining (approximate) optimal controls from the learning problem, with a set of initial/output learning data,  and then running the forward evolution with any given initial data, using the previously determined 'optimal' controls.  The model output is obtained  by applying a final layer affine linear transformation (whose parameters are also determined in the learning process) followed by the final layer activation. If the approximation quality is considered insufficient, then more data sets are added and the learning problem is rerun.

The paper is structured as follows. In Section \ref{section:section_1}, we introduce the concepts of discrete and residual neural networks and discuss the limit procedure that leads to the learning problem, which is formulated as a PDE-constrained optimization problem. In Section  \ref{section:forward_back}, we address the well-posedness of the forward propagation, discuss critical points of the learning task, by computing the gradient of the loss functional, which gives rise to the backward problem. We explore necessary and sufficient conditions for the existence of critical points. Section \ref{section:controllability} covers the controllability of the forward problem, demonstrating that in the single-state case, the system is locally controllable. However, in the multi-state case, controllability is generally not achieved, which points to an instability phenomenon and motivates to constrain the control space. In Section \ref{section:max_principle}, we apply the Pontryagin Maximum Principle to derive necessary conditions for the existence of optimal controls in forward propagation. Finally, in Section \ref{section:HJB}, we examine the value functional associated with forward propagation and the corresponding Hamilton-Jacobi-Bellman equation, proving the existence of viscosity solutions for the latter and establish optimal feedback controls based on the value function.

\section{Discrete Residual Neural Network, Limits and Learning Problem} \label{section:section_1}
A discrete neural network can be described as a recursive function $\Phi : \R^{M_0} \rightarrow \R^{M_L}$, where $M_k \in \N$ represents the number of neurons at each layer $k = 0, \dots, L$. We define $\Phi$ as
\begin{align*}
    \Phi = L_{L} \circ F_{L-1} \circ \dots \circ L_2 \circ F_1 \circ L_1.
\end{align*}
Here $L_k : \R^{M_{k-1}} \rightarrow \R^{M_k}$ is an affine linear map for $k=1, \dots, L$ defined by
\begin{align*}
    L_k(x) = a_k - B_k x,
\end{align*}
where $a_k$ are $M_k$-dimensional vectors called network biases, and $B_k$ are $M_{k} \times M_{k-1}$ matrices called network weights. $F_{k} : \R^{M_k} \rightarrow \R^{M_k}$ is a non linear map given by
\begin{align*}
    F_{k} (\xi) = \left( \sigma (\xi_1), \dots, \sigma (\xi_{M_K}) \right) =: \sigma \left( \xi \right),
\end{align*}
where $\sigma : \R \rightarrow \R$ is the activation function, typically chosen to be a non-decreasing function such as a sigmoid, or a rectified linear unit (ReLU) \cite{berner2021modern}. Here, $\sigma$ acts component-wise, and we slightly abuse notation in the above definition of the map $F_{k}$. One of the most notable properties of the network $\Phi$ is its approximation capabilities. Indeed, it has been proven in \cite{hornik1989multilayer, perekrestenko2018universal} that any continuous function can be approximated with arbitrary accuracy by a multilayer neural network on compact sets, by choosing appropriate weights and biases. This means that for every $f \in C(\R^{M_0})$ and $\forall \varepsilon > 0$, there exists a multilayer network $\Phi$ (constructed as described above) such that $ \| f - \Phi \|_{L^{\infty}(K)} \le \varepsilon$ for $K \subset \R^{M_0}$, where $K$ is a compact set.

Residual neural networks differ slightly from the general neural network model presented above. In this case, $M_k = M$ for all $k = 0, \dots, L$. Denoting the state of the network at layer $k$ by $z_k$, we have
\begin{align*}
    \begin{cases}
    z_{k+1} = z_k + \sigma \left( a_k - B_k z_k \right) \\
    z_0 = x,
    \end{cases}
\end{align*}
with $x \in \R^M$ given.
This resembles the explicit Euler scheme for ODEs, up to a rescaling of the activation function $\sigma \rightarrow \sigma \Delta t$, where $\Delta t << 1$ is the artificially introduced layer width. We underline that residual networks are particularly useful because they prevent the exploding or vanishing gradient problem, which may prevent lower layers from training at all \cite{berner2021modern}. 

A typical high dimensional example for the application of residual networks arises in image processing where unprocessed gray scale $M$ pixel images, each represented by a vector $x \in \R^M$ are mapped into the processed version $z_k(x)$.

The network has width $M$ and depth $L$ at this stage. Setting $T = \Delta t L$ our first goal is to let $\Delta t \rightarrow 0$ while keeping $T$ fixed (which means $L \rightarrow \infty$). This process -- called the infinite depth limit -- will provide us with an ODE system for the state $z$.

With the introduction of the artificial time $t$, we set $t_k = k \Delta t$ and we allow the network biases and weights and the network status to depend on time, i.e., $a^{\Delta t}_k, B_k^{\Delta t}, z^{\Delta t}_k$ where the superscript underlines the dependence on $\Delta t$. Then, we build piecewise linear functions $a^{\Delta t} := a^{\Delta t} (t), B^{\Delta t} := B^{\Delta t}(t) $ by interpolation:
\begin{align*}
    a^{\Delta t}(t_k) = a^{\Delta t}_k, \quad B^{\Delta t} (t_k) = B^{\Delta t}_k, \quad k=0, \dots, L.
\end{align*}
We also define $z^{\Delta t} := z^{\Delta t}(t)$ on $[0, T]$ through the iterative process
\begin{align*}
    z^{\Delta t}(t + \Delta t) &= z^{\Delta t} (t) + \Delta t \sigma \left( a^{\Delta t}(t) - B^{\Delta t}(t) z^{\Delta t}(t) \right) &&\quad 0 \le t \le T - \Delta t \\
    z^{\Delta t}(t) &= x &&\quad 0 \le t \le \Delta t
\end{align*}
with $x = (x_1, \dots, x_M)^{\mathrm{tr}} \in \R^M$.

With suitable hypotheses on the parameters $a^{\Delta t}$, $B^{\Delta t}$ and on the activation function $\sigma$ it is possible to pass to the limit $\Delta t \rightarrow 0$ \cite[Theorem 2.1]{liu2020selection} which leads to the system of $M$ coupled ODEs
\begin{align}
    \label{eq:z_ODE}
    \begin{cases}
        \dot{z} = \sigma \left( a(t) - B(t) z \right) \quad 0 \le t \le T \\
        z(t=0) = x.
    \end{cases}
\end{align}
Here $a(t) = \lim_{\Delta t \rightarrow 0} a^{\Delta t}(t)$, $B(t) = \lim_{\Delta t \rightarrow 0} B^{\Delta t}(t)$. 
We remark that the solution of this problem depends on the labeled datum $x \in \R^M$, i.e., $z(t;x) = z(t) = \left( z_1(t), \dots, z_M(t) \right)^{\mathrm{tr}}$.

In real applications we typically train the network using a large number $N \in \N$ of data sets. We therefore consider a set of initial conditions $\left( x^{(1)}, \dots, x^{(N)} \right)$ and for each data point $x^{(i)}$, we have a system of $M$ nonlinearly coupled ODEs \eqref{eq:z_ODE}, thus producing solutions $\left( z^{(1)}(t), \dots, z^{(N)}(t) \right)$.

We now have to approximate the function $\Phi(x)$ with $z(t;x)$ by choosing appropriate parameter functions $a(t), B(t)$. This process is called supervised learning where we use labeled datasets to train algorithms to predict outcomes and recognize patterns. At this point the parameter functions to be trained in the network are $a(t)$ and $B(t)$.

The next step we want to explore is to take the infinite width limit, i.e., $M \rightarrow \infty$. This is particularly useful as for $M < \infty $ we have several limitations, for instance in image processing \cite{chan2005image}. There, it is very important to analyze geometric features of images (like edges) which is much more intuitive in a continuous framework.
Let us choose $d \in \N$ and $Y \subset \R^d$ an open Jordan set. We partition $Y$ in $M$ disjoint sets such that $\overline{Y} = \bigcup_{k=1}^M \overline{Y_k}$ and for any Lebesgue integrable function $f : Y \rightarrow \R$ we have
\begin{align*}
    \int_{Y} f(y) dy = \sum_{k=1}^M \int_{Y_k} f(y) dy.
\end{align*}
It is worth noting that the label set $Y$ and its dimension $d$ can be chosen freely to contribute to the network architecture and are part of the modeling choices.

The underlying idea of this construction is better understood through the following example. Consider a set of black and white images that we aim to train our network on using labeled data, for example, in image sharpening, denoising, feature extraction. As mentioned before, each vector $x^{(j)} \in \R^M$ represents the gray scale values of a black and white image composed of $M$ pixels. Using the system of ODEs \eqref{eq:z_ODE}, we compute the processed image gray scale values $z^{(j)}(t; x^{(j)})$ for parameter functions $a(t), B(t)$, which will be determined through the training process. 

In this application it makes sense to set $d=2$ and take $Y = (0,1)^2$, the unit square as the image domain. To each $z^{(j)}_k(t)$, we associate a label set (or neuron identifier) $Y_k$ for all $k= 1, \dots, M$ as in Figure \ref{fig:label}. 
An analogous construction can be done with black and white movies, instead of images, where in this case we will choose $d=3$.

\begin{figure}[h]
    \centering
    \includegraphics[width=0.75\linewidth]{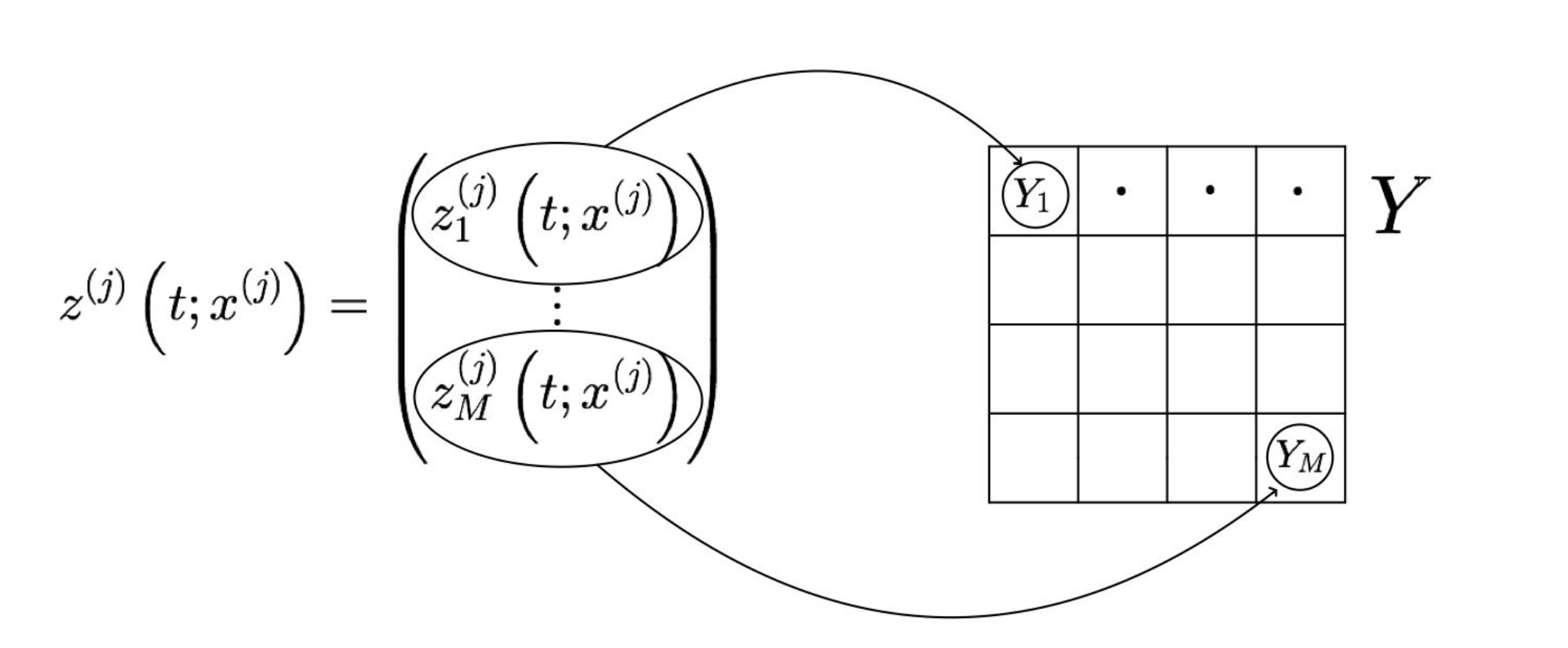}
    \caption{Labeling of network status $z^{j}(t;x^{(j)})$}
    \label{fig:label}
\end{figure}

We now return to the general model and denote the components $a = a(t) \in \mathbb{R}^M$ and $B = B(t) \in \mathbb{R}^{M \times M}$ by:
\begin{align*}
    a(t) = \left( a_1(t), \dots, a_M(t) \right)^{\mathrm{tr}}, \quad B(t) = \left( B_{kl}(t) \right)_{k, l = 1, \dots, M}.
\end{align*}
We define the network bias $a : Y \times [0, T] \rightarrow \mathbb{R}$ and the weight function $b : Y \times Y \times [0, T] \rightarrow \mathbb{R}$ almost everywhere as:
\begin{align*}
    a(y,t) &:= a_k(t) \qquad \quad \text{if} \; y \in Y_k, \\
    b(y, u, t) &:= \frac{1}{|Y_l|} B_{kl}(t) \quad \text{if} \; y \in Y_k, \, u \in Y_l.
\end{align*}
We also define
\begin{align*}
    f(y,t) := z_k(t) \quad \text{if} \; y \in Y_k, \\
    f(y,t=0) =: f_I(y) := x_k \quad \text{if} \; y \in Y_k.
\end{align*}
Note that in this way we have relabeled and 'dimensionalized' the neurons of the network by the variable $y \in Y \subseteq \R^d$.

It is possible to show (see \cite[Theorem 2.2]{liu2020selection}) that the infinite width limit $M \rightarrow \infty$ -corresponding to $\text{diam} \, Y_k \rightarrow 0$ - transforms the ODE system to an integro-differential equation (IDE) for a function $f : Y \times [0, T] \rightarrow \mathbb{R}$, which describes the residual neural network at time $t \in [0, T]$ with neuron identifier $y \in Y$. 
The resulting $N$ integro-differential equations for the training data are given by
\begin{align}
    \label{eq:forward}
    \begin{cases}
        \partial_t f^{(j)} (y, t) = \sigma \left( a(y, t) - \left( B f^{(j)} \right)(y,t) \right) \quad &y \in Y, t \in [0, T]\\
        f^{(j)}(y, t=0) = f_I^{(j)}(y) \quad &y \in Y
    \end{cases}
\end{align}
for all $j = 1, \dots, N$, where $B f^{(j)} (y, t) = \int_{z \in Y} b(y, z, t) f^{(j)}(z, t) dz$ and $\left(f^{(j)}_I\right)_{j=1}^N$ are the (transformed) labeled initial data.

The IDE system described in \eqref{eq:forward} is called the forward problem. It involves modeling and simulating the propagation of data. The forward problem consists of predicting observations given the initial conditions and known model parameters. Once the forward problem is solved, we compute the network output functions $Z^{(j)} : U \rightarrow \R$ as
\begin{align*}
    Z^{(j)} (u) := \int_{Y} w(u,y) f^{(j)}(y, T) dy + \mu (u), \quad \text{for} \; j= 1, \dots, N,
\end{align*}
where the output layer neuron identifier label set $U \subseteq \R^l$ with $l$ possibly different from $d$, $w : U \times Y \rightarrow \R$ and $\mu : U \rightarrow \R$ are terminal weight and bias functions to be also determined during the inverse process. $w$ and $\mu$ are called classifiers.

To finalize the learning problem, we define the predicted outcomes. In many applications the $j$-th predicted outcome only depends locally on the $j$-th network output function:
\begin{align}
    \label{eq:predicted_outcome}
    P^{(j)}_{\text{pre}}(u) = h \left( Z^{(j)} (u)\right),
\end{align}
where $h : \R \rightarrow \R$  is a given prediction function. This step is also known as the regression or classification problem, where the goal is to predict either a function or its class label probabilities. In the former case we choose $h(\xi)=\xi$ on $\R$ and in the latter case a function whose range is the interval $[0, 1]$ is chosen. A common choice for $h$ is the logistic regression function \cite{berner2021modern}
\begin{align*}
    h(\xi) = \frac{e^\xi}{1+e^\xi}
\end{align*}
which converts the output of the network into probabilities of events, associated with labels $u \in U$. Note that the choice of the logistic (i.e., sigmoid) function corresponds to the task of predicting multiple labels for non-exclusive classes so that individual probabilities do not have to sum up to one.

For predicting a single label from multiple classes the soft-max activation \cite[Chapter 6.2.2.3]{goodfellow2016} is often used in the output layer
\begin{align}
\label{eq:soft_max}
    P^{(j)}_{\mathrm{pre}} (u) = \frac{\mathrm{exp} \left( -Z^{(j)}(u) \right)}{ \int_U \mathrm{exp} \left( -Z^{(j)}(v) \right) du(v) },
\end{align}
where $du$ is a bounded Borel measure on $U$, e.g. the atomic measure $du(v) = \sum_{k=1}^L \delta (u_k - v)$. Here $u_1, \dots, u_k \in U$ are finitely many given class labels.

It is important to underline the role of all four training parameters $a, b, w, \mu$. The goal of the learning problem is to estimate these training parameters from observed given label functions $P^{(j)} : U \rightarrow \R$, so that the DNN accurately approximates the data-label relationship for the learning data $\big{\{} f^{(j)}_I, P^{(j)}(u) \big{\}}_{j= 1, \dots, N}$ and generalizes to new unlabeled data.

With this in mind, the learning problem can be recast as an optimization problem:
\begin{align}
    \label{eq:learning_problem}
    \begin{cases}
        \min J(a,b,w,\mu) \\
        \partial_t f^{(j)} (y, t) = \sigma \left( a(y, t) - \left( B f^{(j)} \right)(y,t) \right) \quad &y \in Y, t \in (0, T]\\
        f^{(j)}(y, t=0) = f^{(j)}_I \quad &y \in Y \\
        Z^{(j)} (u) := \int_{Y} w(u,y) f^{(j)}(y, T) dy + \mu (u) \quad &u \in U
        \\
        P^{(j)}_{\text{pre}}(u) \; \mathrm{is} \; \mathrm{given} \; \mathrm{by} \; \eqref{eq:predicted_outcome} \; \mathrm{or} \; \eqref{eq:soft_max} \quad &u \in U,
    \end{cases}
\end{align}
for all $j = 1, \dots, N$, where $J$ is a loss functional measuring the difference between the given (observed) label functions $P^{(j)}$ and those computed by the forward problem, $P^{(j)}_{\text{pre}}$. The aim is to find the "best" parameters $(a,b,w,\mu)$ that minimize the loss functional. This is a data-fitting approach, similar to many other inverse problems formulated as PDE-constrained optimization. Once we have established the setup for the optimization problem \eqref{eq:learning_problem}, we can leverage the powerful techniques from variational calculus and control theory \cite{evans2005introduction} to study its behavior.

Different loss functionals $J$ can be chosen depending on the problem \cite{jadon2024comprehensive, janocha2017on}. In this paper, we will concentrate on the Mean Square Error (MSE) or $L^2$-loss
\begin{align}
    \label{eq:loss}
    J (a, b, w, \mu) := \frac{1}{2N} \sum_{j=1}^N \int_{U} \big{|} P^{(j)}_{\text{pre}} -  P^{(j)} \big{|}^2 du.
\end{align}
This loss functional quantifies the squared difference between the predictions and the target values, assigning a penalty to large deviations from the target value. We reiterate that in many practical classification problems - also in multi-label classification - the measure $du$ is atomic. We note that another widely used loss functional for classification problems is the cross-entropy or log-loss function \cite{masnadi2008design}, which for the single label/multiple class task reads
\begin{align}
    \label{eq:loss_log}
    J(a, b, w, \mu) = - \frac{1}{N} \sum_{j=1}^N \int_{U} \ln \left( \frac{\mathrm{exp} \left( -Z^{(j)}(z) \right)}{ \int_U \mathrm{exp} \left( -Z^{(j)}(v) \right) du(v) } \right) P^{(j)}(z) du(z),
\end{align}
where $P^{(j)} = P^{(j)}(z)$ is the given probability density with respect to the reference measure $du(z)$ on $U$ associated to the $j$-th learning datum $f^{(j)}_I = f^{(j)}_I(y)$.

Note that the cross-entropy 
\begin{align*}
    H(P; Q) := - \int_U P \ln Q d \mu 
\end{align*}
of two probability densities $P, Q$ relative to a reference measure $\mu$ on $U$ assumes its minimum with respect to $Q$ at $P = Q$ such that:
\begin{align*}
    H(P; Q) \ge H(P; P) \quad \forall Q \ge 0 \; \mathrm{with} \; \int_U Q d \mu = 1.
\end{align*}
This follows from the non-negativity of the relative Boltzmann entropy 
\begin{align*}
    E(P; Q) = H(P; Q) - H (P; P) \ge 0
\end{align*}
which is a trivial consequence of Jensen's inequality. Then the loss functional $J$ in \eqref{eq:loss_log} assumes its absolute minimum when 
\begin{align*}
    P^{(j)}_{\mathrm{pre}} (z) = P^{(j)} (z) \quad du(z) \; \mathrm{a.e.}
\end{align*}
Here $P^{(j)}_{\mathrm{pre}}$ is defined in \eqref{eq:soft_max} and the minimal value of $J$ is 
\begin{align*}
    J_{\mathrm{min}} = - \frac{1}{N} \sum_{j=1}^N \int_U \ln (P^{(j)}(z)) P^{(j)}(z) du(z) \ge 0.
\end{align*}

We remark that in the framework of control theory, \eqref{eq:learning_problem} is a fixed time free endpoint problem without running loss, which is commonly referred to as a Mayer problem \cite{liberzon2011calculus}.

For a discussion of appropriate choices for loss functions in classification and regression ML we refer to \cite{janocha2017on} and \cite{jadon2024comprehensive}.

For the sake of a unified presentation, we shall in this paper concentrate on the multiple label/multiple class problem \eqref{eq:learning_problem}, \eqref{eq:predicted_outcome}, \eqref{eq:loss} when considering classification. Also, for the same reason, we shall assume that $du$ is the $l$-dimensional Lebesgue measure on $U$. Generalizations of the theory presented below to other measures on $U$ are straightforward, mostly all it needs is a change of notation.

\section{Well posedness of forward propagation, back propagation and existence of critical points} \label{section:forward_back}
For the coherence of the presentation we begin this section by stating an existence and uniqueness theorem for the forward propagation \eqref{eq:forward}, simplifying the presentation in \cite{liu2020selection}. 

\begin{theorem} \label{theorem:well-posedness_forward}
    Let $f_{I} \in L^2(Y)$, $\sigma \in C^{0,1} (\R)$, $0 < T < \infty$, $|\sigma(0)| |Y| < \infty$, $a \in L^1\left( (0, T); L^2(Y) \right)$ and $b \in L^1\left( (0, T); L^2(Y \times Y) \right)$. Then, the initial-value problem (IVP)
    \begin{align}
        \label{eq:forward_wellpos}
        \begin{cases}
            \partial_t f(y,t) = \sigma \left( a(y,t) - \int_{Y} b(y, z, t) f(z, t) dz \right) \quad &y \in Y, t \in [0, T] \\
            f(y, t=0) = f_I (y) \quad &y \in Y
        \end{cases}
    \end{align}
    has a unique solution in $C \left( [0, T]; L^2(Y) \right)$ which depends uniformly Lipschitz-continuously on the initial data $f_I$ and locally Lipschitz-continuously on the training parameters $a, b$.
    \begin{proof}
        We shall employ the Banach fixed point theorem. Set $X := C \left( [0, T]; L^2(Y) \right)$ and $X_R := \{ f \in X : \| f \|_{X} \le R \}$ for some $0<R< \infty$. We define the operator $ Q_T : X_R \rightarrow X_R$ as
        \begin{align*}
            \left( Q_T f \right) (y, t) := f_I (y) + \int_0^t \sigma \left( a(y, s) - \int_Y b(y, z, s) f(z, s) dz \right) ds
        \end{align*}
         whose fixed points are solutions of \eqref{eq:forward_wellpos}.  Our goal is to demonstrate that $Q_T$ is a contraction on $X_R$ and that $\text{Im}(Q_T) \subset X_R$.

         At first we recall the following standard result from functional analysis \cite{brezis2011functional}. Let $k = k(u,y) \in L^2(U \times Y)$, then the integral operator $\left(K \varphi \right) (u) := \int_Y k(u,y) \varphi(y) dy$ is compact as a map from $L^2(Y)$ into $L^2(U)$ and its $L^2$ operator norm is bounded by the norm of the kernel, i.e., $\| K \| \le \| k \|_{L^2(U \times Y)}$.
         
         Since $\sigma$ is non-decreasing and Lipschitz continuous, we have $0 \le \sigma' \le L$ for some $L >0$ on $\R$. We can then estimate
         \begin{align*}
             \| \left( Q_T f \right) (\cdot, t) \|_{L^2(Y)} \le \| f_I \|_{L^2(Y)} + | \sigma(0) | |Y|^{\frac{1}{2}} t + L \int_0^t \left( \| a(\cdot, s) \|_{L^2(Y)} + \| b (\cdot, \cdot, s) \|_{L^2(Y\times Y)} \| f (\cdot, s) \|_{L^2(Y)} \right) ds.
         \end{align*}
        Choosing $R$ such that $\| f_I \|_{L^2(Y)} + | \sigma(0) | |Y|^{\frac{1}{2}} T + L < \frac{R}{2}$ and $T$ sufficiently small such that $\| a \|_{L^1\left( (0, T); L^2(Y) \right)} \le 1$, $\| b \|_{L^1\left( (0, T); L^2(Y \times Y) \right)} \le \frac{1}{2L}$, we obtain
        \begin{align*}
           \| Q_T f \|_{X_R} \le \frac{R}{2} + \frac{1}{2} \| f \|_{X_R} \le R,
        \end{align*}
        which proves that $\text{Im}(Q_T) \subset X_R$.
        The above construction of $R$ and $T$ leads to
        \begin{align*}
            \| Q_T f_1 - Q_T f_2 \|_{X_R} \le L \| b \|_{L^1\left( (0, T); L^2(Y \times Y) \right)} \| f_1 - f_2 \|_{X_R} \le \frac{1}{2} \| f_1 - f_2 \|_{X_R}.
        \end{align*}
        Thus, $Q_T$ is a contraction in $X_R$ for $T$ sufficiently small. 
        Finally, integrating in time the equation for $f$ and taking its $L^2$ norm, we get
        \begin{align*}
            \| f(t) \|_{L^2(Y)} \le \| f_I \|_{L^2(Y)} + | \sigma(0) | |Y|^{\frac{1}{2}} T + L \| a \|_{L^1 ( (0,t); L^2(Y))} + L \int_{0}^t \| b(s) \|_{L^2(Y \times Y)} \| f(s) \|_{L^2(Y)} ds.
        \end{align*} 
        Consequently, Gronwall's inequality shows that for every $T > 0$ there exists $C = C(T)$ such that $\| f(t) \|_{L^2(Y)} \le C(T)$ for every $t \in [0, T]$. Thus, we proved global existence as $T$ can be extended to $\infty$ (see \cite[Theorem 1.4 pag. 185]{pazy2012semigroups}).

        Similarly, one can show the Lipschitz continuous dependence of the solution $Q_t f$ on the training parameters $a$ and $b$.
    \end{proof}
\end{theorem}

For future reference we state explicitly the estimate for the solution $f$ of \eqref{eq:forward_wellpos}
\begin{align}
    \label{eq:estimate_for_f}
    \| f(\cdot, t) \|_{L^2(Y)} \le \left( \| f_{I} \|_{L^2(Y)} + | \sigma (0) | |Y|^{\frac{1}{2}} t + L \| a \|_{L^1\left( (0,t); L^2(Y) \right)} \right) \mathrm{exp} \left( L \| b \|_{L^1\left( (0,t); L^2(Y \times Y) \right)} \right)
\end{align}
and for the difference of two solutions $f_1, f_2$ with corresponding initial data/training parameters $f_{I, 1}, a_1, b_1$ and $f_{I, 2}, a_2, b_2$ respectively:
\begin{align}
    \label{eq:estimate_1}
    \| f_1 (\cdot, t) - f_2 (\cdot, t) \|_{L^2(Y)} \le &\left( \| f_{I,1} - f_{I, 2} \|_{L^2(Y)} + L \| a_1 - a_2 \|_{L^1 \left( (0,t); L^2(Y) \right)} \right) \exp \left( L \| b_2 \|_{L^1 \left( (0,t); L^2(Y \times Y) \right)} \right) \notag \\
    &+ \left( \| f_{I, 1} \|_{L^2(Y)} + | \sigma (0) | |Y|^{\frac{1}{2}} t + L \| a_1 \|_{L^1\left( (0,t); L^2(Y) \right)} \right) \notag \\
    &\times \| b_1 - b_2 \|_{L^1\left( (0,t); L^2(Y \times Y) \right)} \mathrm{exp} \left( L \| b_1 \|_{L^1\left( (0,t); L^2(Y \times Y) \right)} + \| b_2 \|_{L^1\left( (0,t); L^2(Y \times Y) \right)} \right).
\end{align}

We also state the uniform time continuity estimate for $0 \le t_1 \le t_2 \le T$
\begin{align}
    \label{eq:uniform_time_estimate}
    \| f(\cdot, t_1) - f(\cdot, t_2) \|_{L^2(Y)} \le &| \sigma(0) | |t_1 - t_2| \notag \\
    &+ L \left( \| a \|_{L^1( (t_1, t_2); L^2(Y) )} + \| f \|_{C([t_1, t_2]; L^2(Y))} \| b \|_{L^1 \left( (t_1, t_2); L^2(Y) \times L^2(Y) \right)} \right).
\end{align}

For the remainder of this paper we impose the following assumptions (unless explicitly stated otherwise):
\begin{enumerate}[label=(A\arabic*)]
    \item (label and classification domains) $Y \subseteq \R^d$, $U \subseteq \R^l$ bounded and open , $Y$ and $U$ are equipped with their Lebesgue measures
    \item (activation and classification functions) $\sigma, h : \R \rightarrow \R$ are uniformly Lipschitz-continuous on $\R$
    \item (training data) $\left( f^{(j)}_{I}, P^{(j)} \right) \in L^2(Y) \times L^2(U)$ for $j = 1, \dots, N$ and $f^{(j)}_I \neq f^{(i)}_I$ for $j \neq i$
    \item predicted outcomes $P^{(j)}_{\mathrm{pre}}$ are given by \eqref{eq:predicted_outcome} and the loss functional $J$ by \eqref{eq:loss}.
\end{enumerate}

The existence of minimizers of the task \eqref{eq:learning_problem} depends critically on the choice of the set of controls over which the optimization is performed. The goal is clearly to make that set as large as possible in order to obtain a minimum as small as possible. 
Ideally $\mathcal{S} = \{ (a, b, w, \mu) \in L^1 ((0, T); L^2(Y)) \times L^1 ((0, T); L^2(Y \times Y)) \times L^2 ( U \times Y) \times L^2(U) \}$ is the correct choice. 
However, this would lead to insurmountable mathematical difficulties for proving the existence of a minimizer due to the nonlinearity of $\sigma$ in the forward propagation as well as generic lack of convexity of $J$ in terms of the controls.

Before we shall study potential critical points of $J$ over the space $\mathcal{S}$, we give an existence proof for a minimizer over a rather restricted set for the argmin, based on standard variational techniques.

\begin{theorem} \label{thm:existence_minimizers}
    A minimizer $(a, b, w, \mu)$ of the functional $J$ exists when the minimization is performed over a set $\mathcal{S}_0$ which is compact in the $L^1 ((0, T); L^2(Y)) \times L^1 ((0, T); L^2(Y \times Y)) \times L^2 ( U; L^2(Y) \; \mathrm{weak}) \times L^2(U)$ topology.
    \begin{proof}
        Clearly $0 \le \inf_{(a, b, w, \mu) \in \mathcal{S}_0} J(a, b, w, \mu) < \infty$. Denote its infimum by $l_0$. Then, there exists a minimizing sequence $(a_n, b_n, w_n, \mu_n) \in \mathcal{S}_0$ such that 
        \begin{align*}
            \lim_{n \rightarrow \infty} J(a_n, b_n, w_n, \mu_n) = l_0.
        \end{align*}
        By the compactness of $\mathcal{S}_0$ there exists a subsequence $(a_{n_k}, b_{n_k}, w_{n_k}, \mu_{n_k})$ and $(a_0, b_0, w_0, \mu_0) \in \mathcal{S}_0$ such that
        \begin{align*}
            (a_{n_k}, b_{n_k}, w_{n_k}, \mu_{n_k}) \xrightarrow{n_k \rightarrow \infty} (a_0, b_0, w_0, \mu_0)
        \end{align*}
        in $L^1 ((0, T); L^2(Y)) \times L^1 ((0, T); L^2(Y \times Y)) \times L^2 ( U; L^2(Y) \; \mathrm{weak}) \times L^2(U)$. Then, estimate \eqref{eq:estimate_1} implies that $f^{(j)}_{n_k} \rightarrow f^{(j)}_0$ in $C \left( [0, T]; L^2(Y) \right)$, where $f^{(j)}_{n_k}$ are the forward evolutions associated with $(a_{n_k}, b_{n_k})$ and $f^{(j)}_0$ those associated with $(a_0, b_0)$. Finally 
        \begin{align*}
            \lim_{n_k \rightarrow \infty} J( a_{n_k}, b_{n_k}, w_{n_k}, \mu_{n_k} ) = J (a_0, b_0, w_0, \mu_0 ) = l_0
        \end{align*}
        follows easily using the uniform Lipschitz continuity of the function $h$ on $\R$.
    \end{proof}
\end{theorem}
The restriction imposed by the compactness condition is more severe for the dynamic control parameters $a, b$ than for the output layer regression parameters $w, \mu$. In this context it is interesting to see how the results apply to the finite-dimensional forward evolution discussed in Section \ref{section:section_1}, where the integro-differential equations \eqref{eq:forward} are replaced by ODE-systems of the type in \eqref{eq:z_ODE} with associated learning data $z^{(j)}_I \in \R^M$. Then, by the compact embedding of $\mathrm{BV}(0, T)$ into $L^1(0, T)$ we find that control sets bounded in $\mathrm{BV} \left( (0,T); \R^M \right) \times \mathrm{BV} \left( (0,T); \R^{M^2} \right) \times \R^{M_0 \times M} \times \R^{M_0}$ give the existence of a minimizer via Theorem \ref{thm:existence_minimizers}. Here, $M_0$ denotes the dimension of the output of the final layer of the network. Obviously, countably many jump-discontinuities in time with bounded total jump heights are allowed here.

Successively, our main task is to compute the gradients of the functional $J = J(a,b,w,\mu)$ defined in \eqref{eq:loss} with respect to all of its variables, where $J : L^2\left( Y \times (0,T) \right) \times L^2\left( Y \times Y \times (0,T) \right) \times L^2 (U \times Y) \times L^2 (U) \rightarrow \R$, for the sake of characterising potentially occurring critical points of the functional. We start by computing the first variation of $J$ with respect to $\mu$ in direction $\varphi$, which we will denote as ${\left\langle D_{\mu} J, \varphi \right\rangle}_{L^2(U)}$ with the usual $L^2$ inner product. We compute
\begin{align*}
    {\left\langle D_{\mu} J, \varphi \right\rangle}_{L^2(U)} &= \frac{d}{d \varepsilon} J(a, b, w, \mu + \varepsilon \varphi) \bigg{|}_{\varepsilon=0} \\
    &= \frac{1}{N} \sum_{j=1}^N \int_U \left( P_{\text{pre}}^{(j)}(u) - P^{(j)}(u) \right) h'\left( Z^{(j)}(u) \right) \varphi(u) du,
\end{align*}
which implies
\begin{align}
    \label{eq:D_mu}
    D_{\mu} J (u) = \frac{1}{N} \sum_{j=1}^N \left( P_{\text{pre}}^{(j)}(u) - P^{(j)}(u) \right) h'\left( Z^{(j)}(u) \right).
\end{align}
A similar computations yield the first variation of $J$ with respect to $w$ in direction $v$, i.e.,
\begin{align*}
    {\left\langle D_{w} J, v \right\rangle}_{L^2(U\times Y)} &= \frac{d}{d \varepsilon} J(a, b, w + \varepsilon v, \mu) \bigg{|}_{\varepsilon=0} \\ 
    &= \int_{U} \int_{Y} \left( \frac{1}{N} \sum_{j=1}^N \left( P_{\text{pre}}^{(j)}(u) - P^{(j)}(u) \right) h'\left(Z^{(j)}(u)\right) f^{(j)}(y, T) \right) v(u, y) dy du,
\end{align*}
and
\begin{align}
    \label{eq:D_w}
    D_{w} J (u, y) = \frac{1}{N} \sum_{j=1}^N \left( P_{\text{pre}}^{(j)}(u) - P^{(j)}(u) \right) h'\left(Z^{(j)}(u)\right) f^{(j)}(y, T).
\end{align}
The computation of the first variation with respect to $a$ requires more attention. For this scope, we follow \cite{liu2020selection} and introduce the following notation (to be used in the sequel when useful)
\begin{align*}
    f^{(j)}= f^{(j)}_{a, b}, \quad \xi^{(j)}_{a, b} := a - B_b f^{(j)}_{a, b},
\end{align*}
where $B_b$ is the integral operator with kernel $b = b(y, z, t)$. Then, linearizing $J$ with respect to $a$ in direction $\alpha$, we compute
\begin{align}
    \label{eq:D_a1}
    {\left\langle D_a J, \alpha \right\rangle}_{L^2( (0, T) \times Y)} &= \frac{d}{d \varepsilon} J(a + \alpha, b, w, \mu) \bigg{|}_{\varepsilon=0} \notag \\
    &= \frac{1}{N} \sum_{j=1}^N \int_{Y} \int_U \left( \frac{1}{N} \sum_{j=1}^N \left( P_{\text{pre}}^{(j)}(u) - P^{(j)}(u) \right) h'\left(Z^{(j)}(u)\right) \right) w(u,y) g^{(j)}(y,T) du dy,
\end{align}
where $g^{(j)}$ is the first variation of $f^{(j)}_{a,b}$ with respect to $a$ in direction $\alpha$, i.e.,
\begin{align*}
    g^{(j)} := \lim_{\varepsilon \rightarrow 0} \frac{1}{\varepsilon} \left( f^{(j)}_{a+\varepsilon \alpha, b} - f^{(j)}_{a, b} \right).
\end{align*}
A straightforward computation leads to
\begin{align*}
    \begin{cases}
        \partial_t g^{(j)} = \sigma' \left( \xi^{(j)}_{a, b} \right) \left( \alpha - B_b g^{(j)} \right) \quad &y \in Y, t \in (0, T] \\
        g^{(j)} (y, t=0) = 0 \quad &y \in Y.
    \end{cases}
\end{align*}
Let $M_{a, b}^{(j)}(t, s)$ be the evolution system \cite{pazy2012semigroups} generated by $- \sigma' \left( \xi^{(j)}_{a, b} \right) B_b$, that is $m^{(j)}(t) := M_{a, b}^{(j)}(t, s) m_0$ solves $\partial_t m^{(j)} = - \sigma' \left( \xi^{(j)}_{a, b} \right) B_b m^{(j)}$ for $t \ge s$ and $m^{(j)}(s) = m^{(j)}_0$. Note that $M_{a, b}^{(j)}(t, s)$ is a bounded operator from $L^2(Y)$ into itself, continuous in $t$ and $s$ with respect to the operator norm topology, and it satisfies
\begin{align}
    \label{eq:identity_Mj}
    M^{(j)}_{a, b} (t, s) = I + \int_s^t \Xi^{(j)} (\tau)  M^{(j)}_{a, b} (\tau, s) d\tau,
\end{align}
where $\Xi^{(j)} (t)$ is the integral operator with kernel $- \sigma' \left( \xi^{(j)}_{a, b} (y,t) \right) b(y, z, t)$.

Then
\begin{align*}
    g^{(j)}(y, t) = \int_0^t M^{(j)}_{a, b} (t, s) \left( \sigma' \left( \xi^{(j)}_{a, b} (y, s) \right) \alpha (y, s) \right) ds,
\end{align*}
and $g^{(j)}$ is the Gateaux derivative of $f^{(j)}_{a, b}$ with respect to $a$ in direction $ \alpha$, i.e., $\left( D_a f^{(j)}_{a, b} \right) ( \alpha)$.
To streamline computations we define 
\begin{align*}
    \omega^{(j)}(y) := \int_U \left( P_{\text{pre}}^{(j)}(u) - P^{(j)}(u) \right) h'\left(Z^{(j)}(u)\right) w(u,y) du
\end{align*}
and we substitute the latter expression for $g^{(j)}$ into \eqref{eq:D_a1}, obtaining
\begin{align*}
    {\left\langle D_a J, \alpha \right\rangle}_{L^2( (0, T) \times Y )} &= \frac{1}{N} \sum_{j=1}^N \int_0^T \int_Y \omega^{(j)}(y) M_{a, b}^{(j)}(T, s) \left( \sigma' \left( \xi^{(j)}_{a, b} (y, s) \right) \alpha (y, s) \right) dy ds \\
    &= \frac{1}{N} \sum_{j=1}^N \int_0^T \int_Y M_{a, b}^{(j)}(T, s)^* \left( \omega^{(j)}(y) \right) \sigma' \left( \xi^{(j)}_{a, b} (y, s) \right) \alpha (y, s) dy ds \\
    &=: \frac{1}{N} \sum_{j=1}^N \int_0^T \int_Y r^{(j)}(y, s) \sigma' \left( \xi^{(j)}_{a, b} (y, s) \right) \alpha (y, s) dy ds,
\end{align*}
where $r^{(j)}(y, s) := M_{a, b}^{(j)}(T, s)^* \left( \omega^{(j)}(y) \right)$ solves the following final-value problem
\begin{align}
    \label{eq:backward}
    \begin{cases}
       \partial_t r^{(j)} = \sigma' \left( \xi^{(j)}_{a, b} B_b \right)^* r^{(j)} = B_{b^*} \left( \sigma' \left( \xi^{(j)}_{a, b} \right) r^{(j)} \right) \quad &y \in Y, s \in [0, T) \\
       r^{(j)}(T) = \int_U  \left( P_{\text{pre}}^{(j)}(u) - P^{(j)}(u) \right) h'\left(Z^{(j)}(u)\right) w(u,y) du \quad &y \in Y. 
    \end{cases}
\end{align}
Here, we introduced the notation $*$ to indicate the adjoint of an operator and we used the fact that $B^*_b = B_{b^*}$ with $b^* (y, z, s) = b(z, y, s)$. We conclude
\begin{align}
    \label{eq:D_a}
    D_a J (y, s) = \frac{1}{N} \sum_{j=1}^N \sigma' \left( \xi^{(j)}_{a, b} (y, s) \right) r^{(j)} (y, s).
\end{align}
Finally, the computation of the first variation of $J$ with respect to $b$ follows the exact same structure of the one with respect to $a$. Indeed, we introduce the perturbation $\beta$ in the $b$ direction and define
\begin{align*}
    p^{(j)} (y,t) := \left( D_b f^{(j)}_{a, b} \right) (\beta). 
\end{align*}
The latter satisfies
\begin{align*}
    \begin{cases}
        \partial_t p^{(j)} = - \sigma' \left( \xi^{(j)}_{a, b} \right) \left( B_b p^{(j)} + B_{\beta} f^{(j)}_{a, b} \right) \quad &y \in Y, t \in (0, T] \\
        p^{(j)} (y, t=0) = 0 \quad &y \in Y.
    \end{cases}
\end{align*}
Thus,
\begin{align*}
    p^{(j)}(y, t) = - \int_0^t M^{(j)}_{a, b} (t, s) \left( \sigma' \left( \xi^{(j)}_{a, b}(y, s) \right) \left( B_{\beta} f^{(j)}_{a, b} \right) (y, s) \right) ds.
\end{align*}
Then, we compute
\begin{align*}
    {\left\langle D_b J, \beta \right\rangle}_{L^2( (0, T) \times Y \times Y)} &= \frac{d}{d \varepsilon} J(a, b + \varepsilon \beta, w, \mu) \bigg{|}_{\varepsilon=0} \\
    &= \frac{1}{N} \sum_{j=1}^N \int_{Y} \int_U \left( \frac{1}{N} \sum_{j=1}^N \left( P_{\text{pre}}^{(j)}(u) - P^{(j)}(u) \right) h'\left(Z^{(j)}(u)\right) \right) w(u,y) p^{(j)}(y,T) du dy \\
    &= - \frac{1}{N} \sum_{j =1}^N \int_0^T \int_Y \omega^{(j)}(y) M^{(j)}_{a, b} (T, s) \left( \sigma' \left( \xi^{(j)}_{a, b}(y, s) \right) \left( B_{\beta} f^{(j)}_{a, b} \right) (y, s) \right) dy ds \\
    &= - \frac{1}{N} \sum_{j =1}^N \int_0^T \int_Y r^{(j)}(y, s) \sigma' \left( \xi^{(j)}_{a, b}(y, s) \right) \left( B_{\beta} f^{(j)}_{a, b} \right) (y, s) dy ds \\
    &= - \frac{1}{N} \sum_{j =1}^N \int_0^T \int_Y \int_Y r^{(j)}(y, s) \sigma' \left( \xi^{(j)}_{a, b}(y, s) \right) f^{(j)}_{a, b}(z, s) \beta(y, z, s) dz dy ds.\end{align*}
Consequently
\begin{align}
    \label{eq:D_b}
    D_b J (y, z, s) = - \frac{1}{N} \sum_{j =1}^N f^{(j)}_{a, b}(z, s) \sigma' \left( \xi^{(j)}_{a, b}(y, s) \right) r^{(j)}(y, s).
\end{align}

We collect all the results on the first variations of $J$ in the following Proposition.

\begin{proposition} \label{prop:first_variations}
    The first variations of the loss functional $J \equiv J(a, b, w, \mu)$ \eqref{eq:loss} are
    \begin{align*}
        D_a J (y, s) &= \frac{1}{N} \sum_{j=1}^N \sigma' \left( \xi^{(j)} (y, s) \right) r^{(j)} (y, s) \\
        D_b J (y, z, s) &= - \frac{1}{N} \sum_{j =1}^N f^{(j)}(z, s) \sigma' \left( \xi^{(j)}(y, s) \right) r^{(j)}(y, s) \\
        D_{w} J (u, y) &= \frac{1}{N} \sum_{j=1}^N \left( P_{\text{pre}}^{(j)}(u) - P^{(j)}(u) \right) h'\left(Z^{(j)}(u)\right) f^{(j)}(y, T) \\
        D_{\mu} J (u) &= \frac{1}{N} \sum_{j=1}^N \left( P_{\text{pre}}^{(j)}(u) - P^{(j)}(u) \right) h'\left( Z^{(j)}(u) \right).        
    \end{align*}
    Moreover,
    \begin{align*}
        D J := \left( D_a J, D_b J, D_w J, D_{\mu} J \right) \in L^2\left( Y \times (0, T) \right) \times L^2\left( Y \times Y \times (0, T) \right) \times L^2 (U \times Y) \times L^2 (U)
    \end{align*}
    and $D J$ corresponds to the Gateaux derivative of $J$.
\end{proposition}

From Proposition \ref{prop:first_variations}, necessary and sufficient conditions for a stationary point
\begin{align*}
    \left( a_{\infty}, b_{\infty}, w_{\infty}, \mu_{\infty} \right) \in L^2 \left( Y \times (0, T) \right) \times L^2 \left( Y \times Y \times (0, T) \right) \times L^2 (U \times Y) \times L^2 (U)
\end{align*}
of the functional $J(a, b, w, \mu)$ are
    \begin{align} \label{eq:stationary_a}
       \sum_{j=1}^N \sigma' \left( \xi^{(j)}_{a_{\infty}, b_{\infty}} (y, s) \right) r^{(j)} (y, s) &= 0 \quad a.e. \; y \in Y, s \in (0, T) \\
       \label{eq:stationary_b}
       \sum_{j =1}^N f^{(j)}_{a_{\infty}, b_{\infty}}(z, s) \sigma' \left( \xi^{(j)}_{a_{\infty}, b_{\infty}}(y, s) \right) r^{(j)}(y, s) &= 0 \quad a.e. \; y, z \in Y, s \in (0, T) \\
       \label{eq:stationary_w}
       \sum_{j=1}^N \left( P_{\text{pre}}^{(j)}(u) - P^{(j)}(u) \right) h'\left(Z^{(j)}(u)\right) f^{(j)}_{a_{\infty}, b_{\infty}}(y, T) &= 0 \quad a.e. \; u \in U, y \in Y \\
        \label{eq:stationary_mu}
        \sum_{j=1}^N \left( P_{\text{pre}}^{(j)}(u) - P^{(j)}(u) \right) h'\left(Z^{(j)}(u)\right) &= 0 \quad a.e. \; u \in U,
    \end{align}
where $f^{(j)}_{a_{\infty}, b_{\infty}}$ solve the forward problems \eqref{eq:forward} and $r^{(j)}$ the backward problem \eqref{eq:backward}.

We introduce the concept of weak linear independence, motivated by the structure of \eqref{eq:stationary_a}, \eqref{eq:stationary_b}:
\begin{definition}
    We say that the functions $\{\varphi_j \}_{j=1, \dots, N}$, $\varphi_j : \Omega \subseteq \R^M \rightarrow \R$, are weakly linear independent if $\sum_{j=1}^N \lambda_j \varphi_j = 0$ on $\Omega$ implies $\lambda_j = 0$ for $j = 1, \dots, N$ whenever $\sum_{j=1}^N \lambda_j = 0$.
\end{definition}

The following characterization of weak linear independence is useful.
\begin{proposition} \label{prop:char_wli}
    The functions $\{\varphi_j \}_{j=1, \dots, N}$, $\varphi_j : \Omega \subseteq \R^M \rightarrow \R$ are weakly linear independent if and only if for $J \in {1, \dots, N}$ the $(N-1)$ functions $\{ \varphi_1 - \varphi_J, \dots, \varphi_{J-1} - \varphi_J, \varphi_{J+1} - \varphi_J, \dots, \varphi_N - \varphi_J \}$ are linearly independent.
\end{proposition}

Define the map $T_{F, r} : \R \times L^2(Y) \rightarrow \R$ by
\begin{align}
    \label{eq:Tfr}
    T_{F, r}(a, b) := \sum_{j=1}^N \sigma \left( a - \int_Y b(z) f_j (z) dz \right) r_j
\end{align}
with given parameters $r = \left( r_1, \dots, r_N \right)^{\mathrm{tr}} \in \R^N$ and $F = \left( f_1, \dots, f_N \right)^{\mathrm{tr}} \in L^2(Y)^N$. Define $\lambda_j (a, b) := \sigma' \left( a - \int_Y b(z) f_j (z) dz \right)$, $\Lambda (a, b) := \mathrm{diag} \left( \lambda_1, \dots, \lambda_N \right)$ and compute
\begin{align}
    \label{eq:D_a_TFR} D_a T_{F, r} (a, b) &= \sum_{j=1}^N \lambda_j r_j, \\
    \label{eq:D_b_TFR} D_b T_{F, r} (a, b) &= - \sum_{j=1}^N \lambda_j r_j f_j (y).
\end{align}
Also, denote by $G_F$ the Gram matrix of $\{ f_1, \dots, f_N \}$, i.e., $G_F = \left( g_{ij} \right)_{i,j=1,\dots,N}$ where $g_{ij} = \int_Y f_i (z) f_j (z) dz$.
Clearly $(a,b) \in \R \times L^2(Y)$ is a critical point of $T_{F, r}$ if and only if $D_a T_{F, r} (a, b) = 0$ and $D_b T_{F, r} (a, b) = 0$.

\begin{proposition}
    $(a, b)$ is a critical point of $T_{F, r}$ if and only if
    \begin{align}
        \label{eq:stationary_1} G_F \Lambda(a, b) r &= \bold{0} \\
        \label{eq:stationary_2} e^{\mathrm{tr}} \Lambda(a, b) r &= 0,
    \end{align}
    where $e := (1, \dots, 1)^{\mathrm{tr}}$.
\end{proposition}
\begin{proof}
    Every $f \in L^2(Y)$ can be represented as $f(y) = \alpha^{\mathrm{tr}} F(y) + h(y) $ where $\alpha := \left( \alpha_1, \dots, \alpha_N \right)^{\mathrm{tr}} \in \R^N $ and $h \in \mathrm{span} \{ f_1, \dots, f_N \}^{\perp}$.
    Use this function as multiplier for $D_b T_{F, r}(a, b) = 0$ to obtain after integration over $Y$
    \begin{align*}
        \alpha^{\mathrm{tr}} G_F \Lambda(a, b) r = 0.
    \end{align*}
    Since $\alpha$ is arbitrary we conclude \eqref{eq:stationary_1}. Moreover, $D_a T_{F, r}(a, b) = 0$ can be written compactly as \eqref{eq:stationary_2}.
\end{proof}

Note that $\{ f^{(1)}, \dots, f^{(N)} \}$ is weakly linear independent if and only if the matrix $\begin{bmatrix} G_F \\ e^{\mathrm{tr}} \end{bmatrix}$ has (full) rank $N$. 
Now let $\mathrm{rank} \, G_F = \mathrm{dim} \left( \mathrm{span} \{ f_1, \dots, f_N \} \right) =: K$.
If $\sigma' > 0$ on $\R$, we conclude that $(a, b)$ is a critical point of $T_{F, r}$ if and only if $r$ lies in a linear subspace of $\R^N$ given by the null space of $\begin{bmatrix} G_F \\ e^{\mathrm{tr}} \end{bmatrix} \Lambda(a ,b)$, with dimension $N - \mathrm{rank} \begin{bmatrix} G_F \\ e^{\mathrm{tr}} \end{bmatrix} $, which is $N - K$ if $e \in \mathrm{range} \, G_F$ or $N-K-1$ if $K \le N - 1$ and $e \notin \mathrm{range} \, G_F$.
Clearly, if $\sigma' > 0$ on $\R$ and $\{ f_j \}_{j=1, \dots, N}$ are weakly linear independent, then no critical point exists unless $r = 0$ (which means $T_{F, r} \equiv 0$).

At first we remark that the definition of the co-state $r^{(j)}$ in \eqref{eq:backward} and \eqref{eq:stationary_a} imply that
\begin{align*}
    \sum_{j=1}^N r^{(j)} (y, t) = \sum_{j=1}^N \omega^{(j)} (y), \quad t \in [0, T]
\end{align*}
if $D_a J \left( a_{\infty}, b_{\infty}, w_{\infty}, \mu_{\infty} \right) = 0$, $D_b J \left( a_{\infty}, b_{\infty}, w_{\infty}, \mu_{\infty} \right) = 0$.

Assume now that $\{ f_{I}^{(j)} \}_{j=1, \dots, N}$ are weakly linear independent functions. Since linear independence of functions is stable under small perturbations we conclude from Proposition \ref{prop:char_wli} that weak linear independence is as well. Thus, there exists $ T_1 \in (0, T]$ such that $\{ f^{(j)}_{a_{\infty}, b_{\infty}} (\cdot, t) \}_{j=1, \dots, N}$ is a weakly linear independent set of functions for all $0 \le t \le T_1$. Note that $T_1$ only depends on $\max_{j=1, \dots, N} \| f^{(j)}_I \|_{L^2(Y)} $, on the norm of the inverse of the Gram matrix of $\{ f_{I}^{(j)} - f_{I}^{(J)} \}_{j=1, \dots, N, \, j \neq J}$ and on $a_{\infty}, b_{\infty}$, see the first estimate below Theorem \ref{theorem:well-posedness_forward}. 
Since $a \in L^2 \left( Y \times (0,T) \right)$ and $b \in L^2 \left( Y \times Y \times (0,T) \right)$ we deduce that $a_{\infty} (\cdot, t)$, $b_{\infty}(\cdot, \cdot, t)$ are well defined a.e. for $t \in (0, T)$ with values in $L^2(Y)$ and $L^2(Y \times Y)$ respectively. Therefore $\xi^{(j)}_{a_{\infty}, b_{\infty}} (\cdot, t) $ is well defined for a.e. $t \in (0, T)$ with values in $L^2(Y)$. 
Multiplying \eqref{eq:stationary_a}, \eqref{eq:stationary_b} by a test function $\varphi \in L^2(Y)$ and integrating over $Y$ gives
\begin{align*}
    \sum_{j=1}^N \lambda^{(j)} (t) =0, \; \; \sum_{j=1}^N \lambda^{(j)}(t) f^{(j)}_{a_{\infty}, b_{\infty}} (z, t) = 0 \quad a.e. \, z \in Y, t \in (0, T)
\end{align*}
with $\lambda^{(j)}(t) := \int_Y \sigma' \left( \xi^{(j)}_{a_{\infty}, b_{\infty}} (y, t) \right) r^{(j)} (y, t) \varphi(y) dy$. 
Note that $f^{(j)}_{a_{\infty}, b_{\infty}} \in C \left( [0, T]; L^2(Y) \right)$. Weak linear independence of $\{ f^{(j)}_{a_{\infty}, b_{\infty}}(\cdot, t) \}_{j=1, \dots, N}$ for $0 \le t \le T_1$ gives, since the test function $\varphi$ is arbitrary  in $L^2(Y)$
\begin{align*}
    \sigma' \left( \xi^{(j)}_{a_{\infty}, b_{\infty}} (y, t) \right) r^{(j)} (y, t) = 0 \quad a.e. \, y \in Y, t \in (0, T_1).
\end{align*}
Assume now that $\sigma' > 0$ on $\R$, i.e., $\sigma$ is a strictly increasing activation function, as in the case for the arctan and sigmoid activations (among others). Then $r^{(j)} (y, t) = 0$ a.e. in  $Y \times (0, T_1)$, which together with \eqref{eq:backward} implies
\begin{align*}
    r^{(j)} (y, t) = 0 \quad a.e. \, \mathrm{in} \, Y \times (0, T), j=1, \dots, N,
\end{align*}
and
\begin{align}
    \label{eq:identity_1}
    r^{(j)} (y, T) = \int_{U} \left( P_{\text{pre}}^{(j)}(u) - P^{(j)}(u) \right) h'\left(Z^{(j)}(u)\right) w(u,y) du \equiv 0 \quad a.e. \, y \in Y, j =1, \dots, N.
\end{align}

\begin{remark} \label{remark1}
    If $h(\xi) = \xi$ (regression task), $w(u, y) = \delta(u - y)$, $\mu = 0$ (i.e., $U= Y)$ and, obviously, $J=J(a,b)$ since $w$ and $\mu$ are given and fixed, we immediately conclude from \eqref{eq:identity_1} that $(a_{\infty}, b_{\infty}) \in L^2(Y) \times L^2(Y \times Y)$ is a critical point of $J$ if and only if $\{ P^{(j)} \}_{j=1, \dots, N}$ are reachable from $\{ f_I^{(j)} \}_{j=1, \dots, N}$ through the forward evolution with control parameters $(a_{\infty}, b_{\infty})$.
\end{remark}

\begin{remark} \label{remark:least_square}
    Let $h(\xi) = \xi$, $w \in L^2(U \times Y)$ and $\mu \in L^2(U)$. Then, \eqref{eq:identity_1} implies
    \begin{align*}
        \int_U \left( \int_Y w(u, z) f^{(j)}(z, T) dz + \mu(u) - P^{(j)}(u) \right) w(u, y) du = 0 \quad \mathrm{a.e.} \, y \in Y, j=1, \dots, N,
    \end{align*}
    which we compactly rewrite as
    \begin{align*}
        W^* W f^{(j)} \left( \cdot, T \right) = W^* \left( P^{(j)} - \mu \right), \quad j=1, \dots, N,
    \end{align*}
    where $\left( W \varphi \right)(u) = \int_Y w(u,y) \varphi(y) dy$. We say that for $j=1, \dots, N$, the functions $f^{(j)}(\cdot, T)$ are least-square solutions of "$W f^{(j)}(\cdot, T) = P^{(j)}-\mu$". Note that $f^{(j)}(\cdot, T) \in {\arg \min}_{\xi \in L^2(U)} \| W \xi + \mu - P^{(j)} \|_{L^2(Y)}$. Moreover, since $w \in L^2 (U \times Y)$, $W$ is compact and a compact operator has closed range if and only if it is of finite rank. The latter is thus a sufficient and necessary condition to guarantee the existence of a least-square solution for arbitrary given $P^{(j)}$, $\mu \in L^2(U)$.
\end{remark}

We continue the study of the stationary conditions by analyzing \eqref{eq:stationary_w} and $\eqref{eq:stationary_mu}$.

Let $h(\xi) = \xi$ and denote $F(y) = \left( f^{(1)}(y, T), \dots, f^{(N)}(y, T) \right)^{\mathrm{tr}}$, $P(u) = \left( P^{(1)}(u), \dots, P^{(N)}(u) \right)^{\mathrm{tr}}$ and $\tilde{e} = \frac{e}{\sqrt{N}}$. Multiplying $\eqref{eq:stationary_mu}$ by $F(y) \cdot \frac{e}{N}$ and subtracting it from \eqref{eq:stationary_w}, we obtain
\begin{align}
    \label{eq:w_reduction}
    \int_{Y} w(u, z) F(y)^{\mathrm{tr}}  \left( I - \tilde{e} \otimes \tilde{e} \right) F(z) dz = F(y)^{\mathrm{tr}} \left( I -  \tilde{e} \otimes \tilde{e} \right) P(u) \quad \mathrm{a.e.} \; u, y \in U \times Y.
\end{align}
Let $K:= \mathrm{dim} \left( \mathrm{span} \{ f^{(1)}(\cdot, T), \dots, f^{(N)}(\cdot, T) \} \right)$ and let $\{ \sigma_1, \dots, \sigma_K \} \in L^2(Y)$ be an orthonormal system in $\mathrm{span} \{ f^{(1)}(\cdot, T), \dots, f^{(N)}(\cdot, T) \}$. 
We expand $w(u, y) = \sum_{j=1}^K w_j(u) \sigma_j (y) + h_w(u, y)$, where $\left( w_1, \dots, w_K\right) \in L^2(U)^K$ and $h_w(u, \cdot) \in \mathrm{span} \{ \sigma_1, \dots,  \sigma_K \}^{\perp}$ for a.e. $u \in U$. 
Note that the minimization process \eqref{eq:learning_problem}, \eqref{eq:loss} or \eqref{eq:loss_log} is independent of $h_w$ since it is annihilated in the integral defining $Z^{(j)}$. However, $h_w$ appears in $\omega^{(j)}$ and consequently in $r^{(j)}$ and in the first variations $D_a J, D_b J$ (but not in $D_w J, D_{\mu} J$) as well as in the least-squares formulation of Remark \ref{remark:least_square}. This accounts for the fact that we admit also variations of $f^{(j)}(\cdot, T)$ whose projections on $\{ f^{(1)}(\cdot, T), \dots, f^{(N)}(\cdot, T) \}^{\perp}$ do not vanish.
Substituting the expansion for $w$ into \eqref{eq:w_reduction} yields
\begin{align*}
    \Omega^{\mathrm{tr}} \left( I - \tilde{e} \otimes \tilde{e} \right) \Omega \begin{pmatrix} w_1(u) \\ \vdots \\ w_K(u) \end{pmatrix} = \Omega^{\mathrm{tr}} \left( I - \tilde{e} \otimes \tilde{e} \right) P(u) \quad \mathrm{a.e.} \; u \in U,
\end{align*}
where $\Omega := \left( \Omega_{ij} \right)_{i=1,\dots,N ; j=1, \dots, K}$ with $\Omega_{ij} = \int_{Y} f^{(i)} (y, T) \sigma_j (y) dy$. Defining $A := \left( I - \tilde{e} \otimes \tilde{e} \right) \Omega \in \R^{N \times K}$ we recast the latter equation as
\begin{align}
    \label{eq:w_system}
    A^{\mathrm{tr}} A \begin{pmatrix} w_1(u) \\ \vdots \\ w_K(u) \end{pmatrix} = A^{\mathrm{tr}} \begin{pmatrix} P^{(1)}(u) \\ \vdots \\ P^{(N)}(u) \end{pmatrix} \quad \mathrm{a.e.} \; u \in U.
\end{align}
Note that \eqref{eq:w_system} is a linear system of $K$ equations for a.e. $u \in U$. When $A$ has linearly independent columns, i.e., $\mathrm{rank} A = K$ (which corresponds to the case where $e$ is not in the span generated by the columns of $\Omega$), then $A^{\mathrm{tr}} A$ is invertible and \eqref{eq:w_system} has a unique solution. Otherwise, the rank of $A$ is $K-1$ and \eqref{eq:w_system} is solvable, e.g. by computing the Moore-Penrose inverse \cite{penrose1955generalized} of $A$ denoted by $A^{\dagger}$. Then, a solution of \eqref{eq:w_system} is provided by
\begin{align*}
    \begin{pmatrix} w_1(u) \\ \vdots \\ w_K(u) \end{pmatrix} = A^{\dagger} \begin{pmatrix} P^{(1)}(u) \\ \vdots \\ P^{(N)}(u) \end{pmatrix} \quad \mathrm{a.e.} \; u \in U.
\end{align*}
Finally, from \eqref{eq:stationary_mu} we directly compute $\mu$ as
\begin{align}
    \label{eq:mu}
    \mu(u) = \frac{1}{\sqrt{N}} \left( P(u) \cdot \tilde{e} - \tilde{e}^{\mathrm{tr}} \Omega \begin{pmatrix} w_1(u) \\ \vdots \\ w_K(u) \end{pmatrix} \right) \quad \mathrm{a.e.} \; u \in U.
\end{align}

We summarize the above discussion on the critical points of $J$ in the following Theorem.
\begin{theorem} \label{theorem:w_mu}
    Let $\{ f^{(j)}_I \}_{j=1, \dots, N}$ be weakly linear independent and $h(\xi) = \xi$ for all $\xi \in \R$. 
    Then $\left(a_{\infty}, b_{\infty}, w_{\infty}, \mu_{\infty} \right) \in L^2(Y \times (0, T)) \times L^2(Y \times Y \times (0, T)) \times L^2(U \times Y) \times L^2(U)$ is a critical point of the functional \eqref{eq:loss} if and only if
    \begin{enumerate}
        \item for $j=1, \dots, N$ the functions $f^{(j)}(\cdot, T)$ (i.e., the terminal states of the forward problem \eqref{eq:forward_wellpos}) are least-squares solutions of "$W_{\infty} f^{(j)}(\cdot, T) = P^{(j)} - \mu_{\infty}$" as outlined in Remark \ref{remark:least_square}, \\
        \item $w_{\infty}(u, y) = \sum_{j=1}^K w_{{\infty}_j} (u) \sigma_j (y) + h_{w_{\infty}} (u, y)$, where $\{ \sigma_1, \dots, \sigma_K \} \in L^2(Y)$ is an orthonormal system in $\mathrm{span} \{ f^{(1)}(\cdot, T), \dots, f^{(N)}(\cdot, T) \}$, $h_{w_{\infty}}(u, \cdot) \in \mathrm{span} \{ \sigma_1, \dots,  \sigma_K \}^{\perp}$ with respect to $y$ for a.e. $u$ and $(w_{\infty_1}, \dots, w_{\infty_K} )$ is a $L^2(U)^K$-solution of \eqref{eq:w_system}, \\
        \item $\mu_{\infty}(u)$ is given by \eqref{eq:mu}.
    \end{enumerate}
    
\end{theorem}

The theory developed above does not address the ReLU activation function $\sigma$, as it is not strictly increasing. Henceforth, we discuss this in the following remark.
\begin{remark}
    Let $\sigma : \R \rightarrow \R$ be smooth, non-decreasing and such that $\sigma'(w) = 0$ implies $\sigma(w) =0$ (e.g. a smoothed version of ReLu). 
    Assume that $\{ f_I^{(1)}, \cdots, f_I^{(N)} \}$ are weakly linearly independent. Then, there exists $T_1 \in (0, T]$ such that $\{ f^{(1)}(\cdot, t), \cdots, f^{(N)}(\cdot, t) \}$ are weakly linearly independent for all $t \in [0, T_1]$. We reiterate that for $j=1, \dots, N$, $D_a J = 0$, $D_b J = 0$ implies
    \begin{align*}
        \sigma' \left( \xi^{(j)}_{a_{\infty}, b_{\infty}} (y, t) \right) r^{(j)}(y, t) = 0 \quad \mathrm{a.e.} \, y \in Y, \forall t \in [0, T_1].
    \end{align*}
    This implies $\partial_t r^{(j)} = 0$ a.e. in $Y$, $t \in [0, T_1]$ and $r^{(j)} (y, t) = r^{(j)} (y)$ a.e. in $Y$ for all $t \in [0, T_1]$. Let $r^{(j)}(y) =0$ a.e. in $Y_j \subseteq Y$ and $r^{(j)} \neq 0$ a.e. in $Y^{\mathrm{c}}_j$. Then
    \begin{align*}
        \sigma' \left( \xi_{a_{\infty}, b_{\infty}} (y,t) \right) = 0 \quad \mathrm{a.e.} \, (y,t) \in Y^{\mathrm{c}}_j \times (0, T_1).
    \end{align*}
    From the assumption on $\sigma$ we conclude
        \begin{align*}
        \sigma \left( \xi_{a_{\infty}, b_{\infty}} (y,t) \right) = 0 \quad \mathrm{a.e.} \, (y,t) \in Y^{\mathrm{c}}_j \times (0, T_1),
    \end{align*}
    and $\partial_t f^{(j)}(y, t) = 0$ a.e. in $Y^{\mathrm{c}}_j \times (0, T_1)$ follows. Thus, $f^{(j)}(y, t) = f^{(j)}_I (y)$ a.e. in $Y^{\mathrm{c}}_j \times (0, T_1)$. 
    We conclude that those neurons in the set $Y^{\mathrm{c}}_j$ (the set where $r^{(j)} \neq 0$ ), which are uncharged at $t= 0$, under the evolution of $f^{(j)}$, remain uncharged as long as $\{ f^{(1)}(\cdot, t), \dots, f^{(N)} (\cdot, t) \}$ remain weakly linearly independent.
\end{remark}

\begin{remark}
    Note that there is no uniqueness of optimal controls. To see this, let $(f^{(1)}, \dots, f^{(N)}, a, b, w, \mu)$ be a trajectory of \eqref{eq:learning_problem}. Adding any function $b_1 = b_1(y, z, t)$ to the control component $b$ such that $b_1(y, \cdot, t)$ is for a.e. $y \in Y$ orthogonal to $f^{(1)} (\cdot, t), \dots, f^{(N)} (\cdot, t)$, in $L^2(Y)$ does not change the solutions $f^{(1)}, \dots, f^{(N)}$ of the forward problem \eqref{eq:forward} and therefore gives the same MSE \eqref{eq:loss}.
\end{remark}

This simple observation leads to an interesting reformulation of the forward evolution. We decompose
\begin{align*}
    b(y, z, t) = \sum_{l=1}^N b_{l}(y, t) f^{(l)}(z, t) + b^{\perp}(y, z, t),
\end{align*}
where the first term on the right hand side is for a.e. $y \in Y$ in $\mathrm{span} \{ f^{(1)}(\cdot, t), \dots, f^{(N)}(\cdot, t) \}$ and the second term in its orthogonal complement (with respect to $z$ for a.e. $y \in Y$). Then the forward evolutions \eqref{eq:forward} rewrite as
\begin{align*}
    \partial_t f^{(j)}(y, t) = \sigma \left( a - \sum_{l=1}^N b_l (y, t) \int_{Y} f^{(l)}(z, t) f^{(j)}(z, t) dz \right), \quad j=1, \dots, N,
\end{align*}
with new control parameters $(a, b_1, \dots, b_N) \in L^1 ( Y \times (0, T) )^{N+1}$. The price to pay for this complexity reduction of the control set is that the forward evolution now becomes a fully nonlinearly and non-locally coupled IVP for $( f^{(1)}, \dots, f^{(N)} ) \in C\left( [0, T]; L^2(Y)^N \right)$. Note that different control vectors $(a_1, \vec{b}_1), (a_2, \vec{b}_2)$ give the same forward evolution if and only if $(\vec{b}_1 - \vec{b}_2, a_2 - a_1)$ is in the (at least one-dimensional) kernel of $\begin{bmatrix} G_F \\ e^{\mathrm{tr}} \end{bmatrix}^{\mathrm{tr}}$ for a.e. $y \in Y, t \in (0, T)$. 
In this context the functional \eqref{eq:Tfr} can actually be interpreted as a function from $(a, \vec{b}) \in \R^{N+1}$ into $\R$, with $\vec{b} = (b_1, \dots, b_N)^{\mathrm{tr}}$ and
\begin{align*}
    b(z) = \sum_{j=1}^N b_j f_j (z) + b(z)^{\perp}.
\end{align*}
Then \eqref{eq:stationary_1}, \eqref{eq:stationary_2} become equations for the $(N+1)$ real control parameters.

A powerful mathematical tool to minimize a functional, once its first variation is computed, is the gradient flow technique. For a more in-depth discussion on this topic, we refer to the extensive literature, for instance, \cite{mielke2023introduction, santambrogio2017euclidean}. Let $F : X \rightarrow \R$ be a functional where $X$ is a Hilbert space, in our context $X = L^2$. The gradient flow of $F$ is given by
\begin{align*}
    \begin{cases}
       \dot{x}(\tau) = - D_x F(x(\tau)), \quad \tau > 0 \\
       x (0) = x_0,
    \end{cases}
\end{align*}
where $\tau$ is a time-like variable. Moreover, for $\Delta t_l > 0$ the discrete gradient flow (steepest descent method) is given by
\begin{align*}
    \begin{cases}
        x_{l+1} = x_l - \Delta t_l D_x F(x_l), \quad l = 0, 1, \dots \\
        x_0 \in X.
    \end{cases}
\end{align*}

We remark that the function $F$ is non-increasing along both discrete and continuous gradient flows.

\begin{proposition} \label{prop:stationarity}
    The gradient flow of the loss functional $J \equiv J (a, b, w, \mu)$ \eqref{eq:loss} is
    \begin{align*}
        \frac{\partial \mu(u; \tau)}{\partial \tau} = - D_{\mu} J &= - \frac{1}{N} \sum_{j=1}^N \left( P_{\text{pre}}^{(j)}(u; \tau) - P^{(j)}(u) \right) h'\left( Z^{(j)}(u; \tau) \right) \\
        \frac{\partial w(u, y; \tau)}{\partial \tau} = - D_{w} J &= - \frac{1}{N} \sum_{j=1}^N \left( P_{\text{pre}}^{(j)}(u; \tau) - P^{(j)}(u) \right) h'\left(Z^{(j)}(u; \tau)\right) f^{(j)}_{a, b}(y, T; \tau) \\
        \frac{\partial a(y, s; \tau)}{\partial \tau} = - D_a J &= - \frac{1}{N} \sum_{j=1}^N \sigma' \left( \xi^{(j)}_{a, b} (y, s; \tau) \right) r^{(j)} (y, s; \tau) \\
        \frac{\partial b(y, z, s; \tau)}{\partial \tau} = - D_b J &= \frac{1}{N} \sum_{j =1}^N f^{(j)}_{a, b}(z, s; \tau) \sigma' \left( \xi^{(j)}_{a, b}(y, s; \tau) \right) r^{(j)}(y, s; \tau),
    \end{align*}
    for $\tau >0$, where $\xi^{(j)}_{a, b} := a - B_b f^{(j)}_{a, b}$, $f_{a, b}^{(j)}$ solves the forward problem \eqref{eq:forward_wellpos} and $r^{(j)}$ solves the backward problem \eqref{eq:backward}.
\end{proposition}

The gradient flow updates the control parameter vector $(a, b, w, \mu)$ along the time-like parameter $\tau$. The dependence of the forward and backward propagations on $\tau$ stems solely from their dependence on the control vector.

Note that we need to solve $N$ forward-backward coupled evolution equations in each step of the discrete gradient descent method. The coupling is only 'one-directional', i.e., the $N$ forward problems \eqref{eq:forward} are solved first as they are independent of the backward solutions. The forward solutions $f^{(j)}$ are then used to set up and solve the $N$ back propagations \eqref{eq:backward} for the backward solutions $r^{(j)}$. There is no coupling between forward propagations with different superscripts $(j)$ as well as there is no coupling between different index backpropagations.

\section{Controllability} \label{section:controllability}
In this section, we explore the controllability of the forward problem \eqref{eq:forward}. Specifically, we seek to identify controls that enable the system's initial state to reach a desired terminal state, building on the discussion in Remark \ref{remark1}.
As in the previous section, let $f = f(y,t)$, $a = a(y,t)$ and $ \left( B_b \varphi \right)(y) = \int_Y b(y, z, t) \varphi(z, t) dz$ for $y \in Y$ and $t \in [0, T]$. Let $f$ solve the IVP
\begin{align} \label{eq:forward_controllability}
    \begin{cases}
    \partial_t f = \sigma \left( a - B_b f \right) \\
    f(t=0) = f_I
    \end{cases}
\end{align}
for some control $(a, b)$, and let $f$ satisfy $f(\cdot, T) = \tilde{f}$. We start the discussion with a formal argument to illustrate that 'joint' controllability of $N > 1$ states cannot be expected, even locally around a given state. For $j=1, \dots, N$ we introduce perturbations $f_{I, \epsilon}^{(j)}$, $\tilde{f}_{I, \epsilon}^{(j)}$, where $f_{I, \epsilon}^{(j)} = f_I + O(\epsilon)$ (initial states) and $\tilde{f}_{\epsilon}^{(j)} = \tilde{f} + O(\epsilon)$ (terminal states).

In this section, we aim to address the following question: is there a control $(a_{\epsilon}, b_{\epsilon})$ with $a_{\epsilon} = a + O(\epsilon)$, $b_{\epsilon} = b + O(\epsilon)$ such that for $j=1, \dots, N$ the solutions $f^{(j)}_{\epsilon}$ of
\begin{align*}
    \begin{cases}
    \partial_t f^{(j)}_{\epsilon} = \sigma \left( a_{\epsilon} - B_{b_{\epsilon}} f^{(j)}_{\epsilon} \right) \\
    f^{(j)}_{\epsilon}(t=0) = f^{(j)}_{I, \epsilon}        
    \end{cases}
\end{align*}
satisfy $f^{(j)}_{\epsilon} (\cdot, t=T) = \tilde{f}^{(j)}_{\epsilon}$ and $f^{(j)}_{\epsilon} = f + O(\epsilon)$ for $t \in [0, T]$?

We start by setting, for $j=1, \dots, N$,
\begin{align*}
    f^{(j)}_{\epsilon} &= f + \epsilon g^{(j)} + O(\epsilon^2), \\
    a_{\epsilon} &= a + \epsilon \alpha + O(\epsilon^2), \\
    b_{\epsilon} &= b + \epsilon \beta + O(\epsilon^2), \\
    f^{(j)}_{I, \epsilon} &= f_I + \epsilon g^{(j)}_I + O(\epsilon^2), \\
    \tilde{f}^{(j)}_{\epsilon} &= \tilde{f} + \epsilon \tilde{g}^{(j)} + O(\epsilon^2).
\end{align*}
As before we set $\xi_{a, b}:= a - B_b f$. Clearly, by expansion
\begin{align*}
    \begin{cases}
        \partial_t g^{(j)} = - \sigma'(\xi_{a,b}) B_b g^{(j)} + \sigma'(\xi_{a, b}) \left( \alpha - B_{\beta} f \right) \\
        g^{(j)}(t=0) = g^{(j)}_I.
    \end{cases}
\end{align*}
Denote $h^{(j)} := g^{(j)} - g^{(1)}$, $j=2, \dots, N$. Then $h^{(j)}$ solves
\begin{align*}
    \begin{cases}
        \partial_t h^{(j)} = - \sigma'(\xi_{a, b}) B_b h^{(j)} \\
        h^{(j)}(t=0) = g^{(j)}_I - g_I^{(1)}
    \end{cases}
\end{align*}
and $g^{(j)}(\cdot, t=T) - g^{(1)}(\cdot, t=T) =: h^{(j)}(t=T) = M_{a, b}(T, 0) \left( g_I^{(j)} - g_I^{(1)} \right)$ for $j=2, \dots, N$. Thus, $g^{(j)}(\cdot, t=T) - g^{(1)}(\cdot, t=T) = \tilde{g}^{(j)} - \tilde{g}^{(1)}$ if and only if
\begin{align}
    \label{eq:controllability_condition}
    \tilde{g}^{(j)} - \tilde{g}^{(1)} = M_{a, b}(T, 0) \left( g_I^{(j)} - g_I^{(1)} \right).
\end{align}
The answer to the local multi-state controllability question is generally negative, in particular if \eqref{eq:controllability_condition} does not hold. We shall later on prove a more general version of the above, but at first we turn to the question of single state controllability.

\subsection{Controllability of stationary states} \label{controllability_stationary}
As above we linearize the forward problem \eqref{eq:forward_controllability} with respect to $(f, a, b)$ in direction $(g, \alpha, \beta)$:
\begin{align}
    \label{eq:lin_controllability}
    \partial_t g = \sigma' (a - B_b f) (\alpha - B_{\beta} f - B_b g ) \quad y \in Y, t \in (0, T]
\end{align}
with initial condition $g(t=0) = g_I$ in $Y$. For the following argument we assume that $\sigma(0) = 0$ and $\sigma'(0) >0$. Now, let $\left( f_{\infty}, a_{\infty}, b_{\infty} \right)$ be a stationary solution, i.e., $B_{b_{\infty}} f_{\infty} = a_{\infty}$. For an in-depth discussion of steady states and their stability properties we refer to \cite{liu2020selection}. Then, the linearization reads
\begin{align}
    \label{eq:lin_controllability2}
    \partial_t g = - \sigma' (0) B_{b_{\infty}} g + \sigma' (0) \left( \alpha - B_{\beta} f_{\infty} \right) =: M g + N(\alpha, \beta),
\end{align}
where $M: L^2(Y) \rightarrow L^2(Y)$ is defined by $Mg := - \sigma'(0) B_{b_{\infty}} g$, and $N :L^2(Y) \times L^2(Y \times Y) \rightarrow L^2(Y)$ is defined by $N(\alpha, \beta) := \sigma'(0) \left( \alpha - B_{\beta} f_{\infty} \right)$. We define the resolvent operator (i.e., the evolution semigroup) $R(t_1, t_2) = e^{- \sigma'(0) (t_1 - t_2) B_{b_{\infty}}}$ and the controllability Gramian $\mathcal{G} : L^2(Y) \rightarrow L^2(Y)$ \cite{coron2007} given by
\begin{align*}
    \mathcal{G} := \int_0^T R(T, \tau) N N^* R(T, \tau)^* d\tau.
\end{align*}
It is straightforward to compute $N^* \gamma = \sigma'(0) \left( \gamma, - \gamma \otimes f_{\infty} \right)$, where $\left( \gamma \otimes f_{\infty} \right)(y, z) := \gamma(y) f_{\infty}(z)$. Thus, $N N^* = \sigma'(0)^2 \left( 1 + \int_Y f_{\infty}^2 (z) dz \right) I$ and
\begin{align*}
    \mathcal{G} &= \sigma'(0)^2 \left( 1 + \int_Y f_{\infty}^2 (z) dz \right) \int_0^T R(T, \tau) R(T, \tau)^* d\tau \\
    &= \sigma'(0)^2 \left( 1 + \int_Y f_{\infty}^2 (z) dz \right) \int_0^T e^{- \sigma'(0) (T- \tau) B_{b_{\infty}}} e^{- \sigma'(0) (T- \tau) B_{b_{\infty}^*}} d\tau.
\end{align*}

Note that $\mathcal{G}$ is self adjoint on $L^2(Y)$. We now prove that $\mathcal{G}$ is coercive. We compute for $r_0 \in L^2(Y)$
\begin{align*}
    \left( r_0, \mathcal{G} r_0 \right)_{L^2(Y)} = \sigma'(0)^2 \left( 1 + \int_Y f_{\infty}^2 (z) dz \right) \int_0^T \| r(t) \|^2_{L^2(Y)} dt = 0,
\end{align*}
where $r(t) = e^{- \sigma'(0) t B_{b^*_{\infty}}} r_0$. Since $B_{b_{\infty}}$ and its adjoint are bounded
\begin{align*}
    \| e^{- \sigma'(0) t B_{b^*_{\infty}}} \| \le e^{\sigma'(0) |t| \| b^*_{\infty} \|_{L^2(Y \times Y)}} \le \sqrt{C(T)}, \quad 1 \le C(T) < \infty
\end{align*}
and we conclude that for every $r_0 \in L^2(Y)$
\begin{align}
    \| r_0 \|^2_{L^2(Y)} = \| e^{\sigma'(0) t B_{b_{\infty}^*}} r(t) \|^2_{L^2(Y)} \le C(T) \| r(t) \|^2_{L^2(Y)}.
\end{align}
Thus,
\begin{align*}
    \int_0^T \| r(t) \|^2_{L^2(Y)} dt \ge \frac{1}{C(T)} \| r_0 \|^2_{L^2(Y)}.
\end{align*}
Consequently, \cite[Theorem 2.42]{coron2007} implies that the linearized problem \eqref{eq:lin_controllability2} is exactly controllable in $L^2(Y)$. This implies - for finite rank operators $B$ as in \eqref{eq:z_ODE} (where $L^2(Y)$ is replaced by $\R^M$ and $B(t)$ is a $M \times M$ matrix) - that the nonlinear problem \eqref{eq:forward} is small-time locally controllable at an equilibrium (see \cite[Theorem 3.8]{coron2007}).

\subsection{Controllability of general states}
Let $f =f(y, t)$, $a = a(y, t)$ and $b = b(y, z, t)$ be a trajectory of \eqref{eq:forward_controllability}. Consider \eqref{eq:lin_controllability} for $\alpha = \alpha(y, t)$ and $\beta = \beta(y, z, t)$, then
\begin{align}
    \label{eq:controll_lin_general}
    \partial_t g = -\sigma' \left( \xi_{a, b} \right) B_b g + \sigma' \left( \xi_{a, b} \right) \left( \alpha - B_{\beta} f \right) =: M(t) g + N(t)(\alpha, \beta)
\end{align}
with initial condition $g(t=0) = g_I$ in $Y$, where $M(t): L^2(Y) \rightarrow L^2(Y)$ is defined by $M(t) g := - \sigma'\left( \xi_{a, b} \right) B_{b} g$, and $N(t) :L^2(Y) \times L^2(Y \times Y) \rightarrow L^2(Y)$ is defined by $N(t)(\alpha, \beta) := \sigma'\left( \xi_{a, b} \right) \left( \alpha - B_{\beta} f \right)$.

An analogous computation to the one carried in Section \ref{controllability_stationary} yields
\begin{align*}
    N(t) N(t)^* = \delta^2(y, t) I,
\end{align*}
where $\delta^2 (y, t) := \left( \sigma' \left( \xi_{a, b} \right) \right)^2  \left( 1 + \int_Y f(z, t)^2 dz \right)$. The controllability Gramian reads
\begin{align*}
    \mathcal{G} = \int_0^T M_{a, b}(t, \tau) N(t) N(t)^* M_{a, b}(t, \tau)^* d\tau
\end{align*}
where
\begin{align*}
    \begin{cases}
        \partial_t M_{a, b}(t, \tau) = M(t) M_{a, b}(t, \tau) \\
        M_{a, b}(t = \tau, \tau) = I.
    \end{cases}
\end{align*}

\begin{proposition}
    Let $f, a \in L^{\infty} \left( Y \times (0, T) \right)$, $b \in L^{\infty} \left( Y \times Y \times (0, T) \right)$. Let $\sigma'$ be bounded above on $\R$ and below away from $0$ uniformly on bounded sets of $\R$. Then \eqref{eq:controll_lin_general} is controllable.
    \begin{proof}
        From the properties of the evolution system $M_{a, b}(t, \tau)$ and the boundedness of $B_b$ and $B_{b^*}$ we find that there exists $C_0 > 0$ such that $\left( r_0, \mathcal{G} r_0 \right) \ge C_0 \| r_0 \|^2_{L^2(Y)}$ which gives coercivity of $\mathcal{G}$ and \cite[Theorem 2.42]{coron2007} yields the result.
    \end{proof}
\end{proposition}
From \cite[Theorem 3.6]{coron2007} we conclude that the nonlinear problem \eqref{eq:forward_controllability} is locally controllable along the trajectory $(f, a, b)$ in time $T$ for the finite dimensional case as presented in Section \ref{section:section_1}.

We remark that it is not difficult to eliminate the finite rank assumption on $B_b$ by a straightforward limit procedure. 

We now return to the initial question of this section, namely whether we can control the forward problem using a single control $(a, b)$ for $N > 1$ pairs of initial and target functions within a given time $T > 0$. In other words, can there exist a control $(a, b) \in L^1 ((0, T); L^2(Y)) \times L^1 ((0, T)); L^2 (Y \times Y))$ such that the solutions $f^{(j)} = f^{(j)}(y, t)$ of
\begin{align*}
    \begin{cases}
        \partial_t f^{(j)} = \sigma \left( a - B_b f^{(j)} \right) \quad &y \in Y, t \in (0, T] \\
        f^{(j)}(t=0) = f_I^{(j)} (y) \quad &y \in Y,
    \end{cases}
\end{align*}
for $j=1, \dots, N$ and given $\{ f_I^{(1)}, \dots, f_I^{(N)} \}$, satisfy $f^{(j)}(y, t=T) = \tilde{f}^{(j)} (y)$ a.e. in $Y$? The answer in general is no (in a stable way).

Let us assume that $| \sigma' | \le L$ in $\R$, then from estimate \eqref{eq:estimate_1} we get
\begin{align*}
    \| f^{(j)} (t) - f^{(1)} (t) \|_{L^2(Y)} \le \| f^{(j)}_I - f^{(1)}_I \|_{L^2(Y)} \mathrm{exp} \left( L \| b \|_{L^1((0, T); L^2(Y \times Y))} \right).
\end{align*}
Moreover,
\begin{align*}
    \partial_t \left( f^{(j)} - f^{(1)} \right) &= \sigma(\xi_{a,b}^{(j)}) - \sigma(\xi_{a,b}^{(1)}) \\
    &= - \sigma'(\xi^{(1)}_{a, b}) B_b \left( f^{(j)} - f^{(1)} \right) + \Phi
\end{align*}
where $\Phi(y, t) := \frac{\sigma''(\xi_j)}{2} \left( B_b \left( f^{(j)} - f^{(1)} \right) \right)^2$ assuming $\sigma'' \in L^{\infty}(\R) \cap C(\R)$. Here $\xi_j$ is an intermediate value between $\xi_{a,b}^{(j)}$ and $\xi_{a,b}^{(1)}$. Then
\begin{align*}
    \| \Phi (\cdot, t) \|_{L^2(Y)}^2 &\le C \int_Y \left( \int_Y |b(y, z, t)| \big{|} f^{(j)}(z, t) - f^{(1)}(z, t) \big{|} dz \right)^4 dy \\
    &\le C \int_Y \left( \int_Y |b(y, z, t)|^2 dz \right)^2 dy \| f^{(j)}(t) - f^{(1)}(t) \|^4_{L^2(Y)} \\
    &\le C \int_Y \left( |Y|^{\frac{1}{2}} \left( \int_Y |b(y, z, t)|^4 dz \right)^{\frac{1}{2}} \right)^2 \| f^{(j)}(t) - f^{(1)}(t) \|^4_{L^2(Y)} \\
    &= C |Y| \| b(t) \|_{L^4(Y \times Y)}^4 \| f^{(j)}(t) - f^{(1)}(t) \|^4_{L^2(Y)}
\end{align*}
for some constant $C > 0$, which leads to
\begin{align*}
    \| \Phi (\cdot, t) \|_{L^2(Y)} \le C |Y|^{\frac{1}{2}} \| b(t) \|_{L^4(Y \times Y)}^2 \| f^{(j)}(t) - f^{(1)}(t) \|^2_{L^2(Y)}.
\end{align*}
Now let $M^{(1)}_{a, b}(t, \tau)$ be the evolution system generated by $-\sigma' \left( \xi_{a,b}^{(1)} \right) B_b$ introduced in Section \ref{section:forward_back}. Then
\begin{align}
    \label{eq:link_target-initial}
    \tilde{f}^{(j)}-\tilde{f}^{(i)} = M^{(1)}_{a, b}(T, 0) \left( f^{(j)}_{I} - f^{(1)}_I \right) + \nu(y)
\end{align}
where $\nu(y) := \int_0^T \left( M^{(1)}_{a, b}(t, \tau) \Phi(\cdot, \tau) \right)(y) d\tau$. From \eqref{eq:identity_Mj} we get
\begin{align*}
    \| M_{a, b}^{(1)} (t, s) \| \le 1 + L \int_s^t \| b(\tau) \|_{L^2(Y \times Y)} \| M^{(1)}_{a, b} (\tau, s) \| d\tau 
\end{align*}
and through Gronwall's inequality, we obtain
\begin{align*}
    \| M_{a, b}^{(1)} (t, s) \| \le \mathrm{exp} \left( L \| b \|_{L^1 \left( (s, t); L^2(Y \times Y) \right)} \right).
\end{align*}
Thus, 
\begin{align*}
    \| \nu \|_{L^2(Y)} &\le \int_0^T \mathrm{exp} \left( L \| b \|_{L^1 \left( (\tau, T); L^2(Y \times Y) \right)} \right) \| \Phi(\tau) \|_{L^2(Y)} d\tau \\
    &\le C_Y \mathrm{exp} \left( L \| b \|_{L^1 \left( (0, T); L^2(Y \times Y) \right)} \right) \int_0^T \| b(\tau) \|^2_{L^4(Y \times Y)} \| f^{(j)}(\tau) - f^{(1)}(\tau) \|^2_{L^2(Y)} d\tau \\
    &\le C_Y \mathrm{exp} \left(2 L \| b \|_{L^1 \left( (0, T); L^2(Y \times Y) \right)} \right) \| b \|^2_{L^2((0, T); L^4(Y \times Y))} \| f_I^{(j)} - f^{(1)}_I \|_{L^2(Y)}^2,
\end{align*}
where $C_Y > 0$ is a constant that depends on $Y$. Now choose a sequence of initial data $\{ f^{(j)}_{I, \epsilon} \}_{j=1}^N \in L^2(Y)^N$ and a sequence of terminal data $\{ \tilde{f}^{(j)}_{\epsilon} \}_{j=1}^N \in L^2(Y)^N$ such that
\begin{enumerate}
    \item $f^{(1)}_I, \tilde{f}^{(1)}$ are independent of $\epsilon$.
    \item For all $j=2, \dots, N$, $f^{(j)}_{I, \epsilon} - f^{(1)}_I = \epsilon g^{(j)}_I \neq 0 $ with $g^{(j)}_I \in L^2(Y)$ and $\tilde{f}^{(j)}_{\epsilon} - \tilde{f}^{(1)} = \epsilon \tilde{h}^{(j)}_{\epsilon} \neq 0 $ with $\tilde{h}^{(j)}_{\epsilon} \in L^2(Y)$ uniformly in $\epsilon$.
\end{enumerate}
From \eqref{eq:link_target-initial} we conclude that
\begin{align*}
    \tilde{h}^{(j)}_{\epsilon} = M_{a, b}^{(1)}(T, 0) g^{(j)}_I + \frac{\nu}{\epsilon} \quad j=2, \dots, N
\end{align*}
and
\begin{align*}
    \bigg{\|} \frac{\nu}{\epsilon} \bigg{\|}_{L^2(Y)} \le C_Y \epsilon \mathrm{exp} \left(2 L \| b \|_{L^1 \left( (0, T); L^2(Y \times Y) \right)} \right) \| b \|^2_{L^2((0, T); L^4(Y \times Y))} \| g_I^{(j)} \|^2_{L^2(Y)}.
\end{align*}
Choose, e.g. $\tilde{h}^{(j)}_{\epsilon} = \frac{1}{2} M^{(1)}_{a, b}(T, 0) g^{(j)}_I$. Then
\begin{align*}
    g_I^{(j)} = -2 M^{(1)}_{a, b}(T, 0)^{-1} \frac{\nu}{\epsilon} = O(\epsilon)
\end{align*}
if $\| b \|_{L^2( (0,T); L^4(Y \times Y) )} = O(1)$. Thus, there is no control $(a_{\epsilon}, b_{\epsilon})$ with $\| b_{\epsilon} \|_{L^2( (0,T); L^4(Y \times Y) )} = O(1)$ which takes $f^{(j)}_{I, \epsilon}$ into $\tilde{f}^{(j)}_{\epsilon}$. This is obviously an instability phenomenon.

In conjunction with Theorem \ref{theorem:w_mu} this is clearly a negative result for the regression task since in the case of a strictly increasing activation function it excludes the existence of $O(1)$ minimizers of the loss functional $J$ for general weakly linearly independent $O(1)$ training initial data and general $O(1)$ target data. An efficient and commonly used method to tackle this issue is to confine the control functions, either by Tikhonov regularisation of the loss functional (see \cite{liu2020selection}) or by simply imposing pointwise bounds for the control functions $a$ and $b$. The latter approach will be analyzed in detail in the next Section.

\section{Constraints on the controls - The Pontryagin minimum principle} \label{section:max_principle}
We now approach the deep learning problem within the framework of mathematical control theory \cite{evans2005introduction, zabczyk2020mathematical}, following the Pontryagin Minimum Principle (commonly referred to in the literature as the Pontryagin Maximum Principle) as described in \cite{lewis2006maximum, pontryagin2018mathematical}. 
Although the original principle searches for maxima, the same reasoning obviously applies to minima by simply reversing the sign of the loss functional. The Pontryagin Minimum Principle is of particular interest when optimizer parameters vary in regions with boundaries. 
In this context, we aim to find optimal controls for the parameters $a$ and $b$ in the forward problem, assuming that $w$ and $\mu$ are given. 
Thus, the learning task is formulated as a minimization problem, where we look for optimal learning parameters $(\bar{a}, \bar{b})$ of the forward problem \eqref{eq:forward} that minimize the loss functional \eqref{eq:loss}, i.e.,
\begin{align}
    \label{eq:minimization_problem}
    \min_{(a,b) \in \mathscr{C}} J(a, b) = J(\bar{a}, \bar{b}).
\end{align}
Here, $\mathscr{C}$ is the control set, defined by
\begin{align*}
    \mathscr{C} := \big{ \{ } &(a, b) : a \in L^{\infty} \left( (0, T); L^{\infty}(Y) \right), b \in L^{\infty} \left( (0, T); L^{\infty}(Y) \times L^{\infty}(Y) \right), \\ 
    &\left( a(y, t), b(y, z, t) \right) \in A \; \mathrm{a.e.} \; \mathrm{in} \; y, z \in Y, t \in (0, T) \big{ \} },
\end{align*}
where $A$ is a bounded, convex and closed subset of $\R^2$ with a non-empty interior. Also we introduce the set:
    \begin{align}
        \mathscr{A}_{\mathrm{HJB}} := \big{ \{ } (a, b) \in L^{\infty}(Y) \times L^{\infty}(Y \times Y) : \left( a(y), b(y, z) \right) \in A \; \mathrm{for} \; \mathrm{a.e.} \; y \in Y, z \in Y \big{ \} },
    \end{align}
(the notation will be self-explanatory in the next Section).

We define the state variable vector $F(y, t) := \left( f^{(1)}(y, t), \dots, f^{(N)}(y, t) \right)^{\mathrm{tr}}$, $F_I(y) := \left( f^{(1)}_I (y), \dots, f^{(N)}_I (y) \right)^{\mathrm{tr}}$ which solves
\begin{align*}
    \begin{cases}
        \partial_t F = \sigma \left( a e - B_b F \right), \\
        F(y, t=0) = F_I(y),
    \end{cases}
\end{align*}
with the obvious abuse of notation $\sigma(v) = \left( \sigma(v_1), \dots, \sigma(v_N) \right)^{\mathrm{tr}}$ for a vector $v \in R^N$.

We also introduce the co-state variable vector $r = \left( r^{(1)}(y, t), \dots, r^{(N)}(y, t) \right)^{\mathrm{tr}}$ as a solution of the backward problem \eqref{eq:backward} and define the control-theory Hamiltonian
\begin{align}
    \label{eq:hamiltonian}
    H(F, r, a, b) = \sum_{j=1}^N \int_Y \sigma \left( a - B_b f^{(j)} \right) r^{(j)} dy,
\end{align}
for $F, r \in L^2(Y)^N$ and $(a, b) \in \mathscr{C}$. 
The minimum principle states that -- see \cite{bardi1997optimal} -- for an optimal control $(\bar{a}, \bar{b}) \in \mathscr{C}$ and the corresponding trajectory $\bar{F} = \left( \bar{f^{(1)}}, \dots, \bar{f^{(N)}} \right)^{\mathrm{tr}}$ there exists a co-state $\bar{r} = \left( \bar{r^{(1)}}, \dots, \bar{r^{(N)}} \right)^{\mathrm{tr}}$ such that a.e. in $t \in (0, T)$
\begin{align}
    \label{eq:control_hamiltonian}
    H(\bar{F}, \bar{r}, \bar{a}, \bar{b}) &= \sum_{j=1}^N \int_Y \sigma \left( \bar{a} - B_{\bar{b}} \bar{f^{(j)}} \right) \bar{r^{(j)}} dy \notag \\
    &= \min_{(a,b) \in \mathscr{A}_{\mathrm{HJB}}} H(\bar{F}, \bar{r}, a, b) \notag \\
    &= \int_Y \min_{\left( a, b \right) \in \mathscr{A}} \left( \sum_{j=1}^N \sigma \left( a - \int b(z) \bar{f^{(j)}} (z, t) dz \right) \bar{r^{(j)}}(y, t) \right) dy \notag \\
    &= \int_Y \min_{(a, b) \in \mathscr{A}} T_{F(\cdot, t), r(y, t)} (a, b) dy
\end{align}
where we use the notation introduced in \eqref{eq:Tfr} and $\mathscr{A} := \{ (a, b) \in \R \times L^{\infty}(Y) : (a, b(z) ) \in A \; \mathrm{a.e.} \; \mathrm{in} \; Y \}$ is closed, bounded and convex in $\R \times L^2(Y)$. The first equality in \eqref{eq:control_hamiltonian} stems from the direct application of the minimum principle while the second one requires close scrutiny, for which we shall proceed in two steps. 
We prove:
\begin{proposition} \label{proposition_6}
    \begin{enumerate}
        \item \label{first_Statement} $T_{F, r}$ assumes its minimum on $\mathscr{A}$.
        \item \label{second_Statement} Let $\{ f_j \}_{j=1, \dots, N}$ be weakly linear independent. Assume that $\sigma \in C^1(\R)$ and $\sigma' > 0$ on $\R$. Then $T_{F, r}$ assumes its minimum on $\partial \mathscr{A}$ unless $r_j = 0$ for all $j=1, \dots, N$.
    \end{enumerate}
    We remark that in (\ref{second_Statement}) the boundary of $\mathscr{A}$ is understood with respect to the $\R \times L^{\infty}(Y)$ topology.
    \begin{proof}
        Since $\sigma$ is locally bounded on $\R$ we conclude that there exists $m < \infty$ such that
        \begin{align*}
            \inf_{(a,b)\in \mathscr{A}} T_{F, r}(a, b) = m.
        \end{align*}
        Now, let $(a_n, b_n) \in \mathscr{A}$ be a minimizing sequence, i.e., $\lim_{n \rightarrow \infty} T_{F, r}(a_n, b_n) = m$. Since $\mathscr{A}$ is closed and convex in $\R \times L^2(Y)$, it follows from Mazur's theorem that it is weakly closed in $\R \times L^2(Y)$. Additionally, the boundedness of $\mathscr{A}$ allows us to use Eberlein-Šmulian theorem \cite{whitley1967elementary} which guarantees its weak compactness in $\R \times L^2(Y)$. Therefore, there exists a subsequence $(a_{n_k}, b_{n_k})$ such that $ (a_{n_k}, b_{n_k}) \weakconv (a, b) \in \mathscr{A}$ in $\R \times L^2(Y)$ as $n_k \rightarrow \infty$. Since the operator $T_{F, r}$ is continuous in the weak topology of $\R \times L^2(Y)$, we conclude $T_{F, r}(a, b) = m$, which proves (\ref{first_Statement}). 

        To prove (\ref{second_Statement}) we assume that $T_{F, r}$ attains its minimum at $(a, b)$ in the interior of $\mathscr{A}$, namely $\mathring{\mathscr{A}}$, with respect to the $\R \times L^{\infty}(Y)$ topology. Then, $D_a T_{F, r} (a, b) = 0$ and $ D_b T_{F, r} (a, b) = 0$ computed in \eqref{eq:D_a_TFR}, \eqref{eq:D_b_TFR} where $D_a T_{F, r}, D_b T_{F, r}$ denote the Gateaux derivatives. The weak linear independence of $\{ f_j \}_{j=1, \dots, N}$ implies
        \begin{align*}
            \sigma' \left( a - \int_Y b(z) f_j (z) dz \right) r_j = 0 \quad \forall j=1, \dots, N.
        \end{align*}
        Since $\sigma'>0$, the latter identity implies $r_j = 0$ for all $j= 1, \dots, N$. Thus, we conclude that either $T_{F, r}$ attains its minimum on the boundary $\partial \mathscr{A}$ or $r = 0$.
    \end{proof}
\end{proposition}

Obviously,
\begin{align*}
    \min_{(a, b) \in \mathscr{C}} H \left( \bar{F}(\cdot, t), \bar{r}(\cdot, t), a, b \right) \ge \int_Y \min_{(a,b) \in \mathscr{A}} T_{\bar{F}(\cdot, t), \bar{r}(y, t)}(a, b) dy.
\end{align*}
To establish equality, which shows the second equality in \eqref{eq:control_hamiltonian}, we now prove a result of the Borel-measurable selection of optimal controls.
\begin{proposition} \label{prop:measurability}
    Let $\tilde{F}, \tilde{r} \in L^1 \left( (0, T); L^2(Y)^N \right)$. Then there exists a Borel-measurable map $(a^*, b^*) : \left( Y \times (0,T) \right) \times \left( Y \times Y \times (0,T) \right) \rightarrow A$ such that $\left( a^* (y, t), b^*(y, \cdot, t) \right) \in \mathrm{argmin} T_{\tilde{F}(\cdot, t), \tilde{r}(y, t)} (a, b)$ a.e. in $(Y \times (0, T))^2$.
\end{proposition}
\begin{proof}
    We note that the map $(F, r, a, b) \rightarrow T_{F, r}(a, b)$ in $L^2(Y)^N \times \R^n \times \R \times \left( L^2(Y)-\mathrm{weak} \right)$ is continuous and start by invoking Proposition 7.33 in \cite{bertsekas1996stochastic}. Following the notation in that reference we set $X_0 = L^2(Y)^N \times \R^N$, $Y_0 = \R \times \left( L^2(Y)-\mathrm{weak} \right) \cap \mathscr{A}$. Note that $Y_0$ is compact and metrizable since the weak topology on bounded subsets of $L^2(Y)$ is metrizable. With $D = X_0 \times Y_0$ we conclude the existence of a Borel-measurable map $\varphi : X_0 \rightarrow Y_0$ such that
    \begin{align*}
        T_{F, r}\left( \varphi(F, r) \right) = \min_{(a, b) \in \mathscr{A}} T_{F, r}(a, b).
    \end{align*}
    We now define $(a^*, b^*) = \varphi \circ (\tilde{F}, \tilde{r})$ and the result follows since the composition of Borel-measurable maps is Borel-measurable.
\end{proof}

If $A = [a_m, a_M] \times [b_m, b_M]$ then
\begin{align*}
    \partial \mathscr{A} = &\{ (a, b) \in \mathscr{A} : a = a_m \; \mathrm{or} \; a = a_M \; \mathrm{and} \; b_m \le b(z) \le b_M \; \mathrm{a.e.} \; \mathrm{in} \; Y \} \\
    &\cup \{ (a, b) \in \mathscr{A} : a_m \le a \le a_M, b_m \le b(z) \le b_M \; \mathrm{a.e.} \; \mathrm{in} \; Y \; \mathrm{and} \; \operatorname*{ess\,inf}_{z \in Y} b(z) = b_M \; \mathrm{or} \; \operatorname*{ess\,sup}_{z \in Y} = b_M \}.
\end{align*}

\begin{remark}
    If $\sigma' > 0$ and $\{ f_j\}_{j=1, \dots, N}$ are linearly independent and if $r \neq 0$ then $T_{F, r}$ assumes its minimum at the following subset of $\partial \mathscr{A}$
    \begin{align*}
        \{ (a, b) \in \mathscr{A} : a_m \le a \le a_M, b_m \le b(z) \le b_M \; \mathrm{a.e.} \; \mathrm{in} \; Y \; \mathrm{and} \; \operatorname*{ess\,inf}_{z \in Y} b(z) = b_M \; \mathrm{or} \; \operatorname*{ess\,sup}_{z \in Y} = b_M \}.
    \end{align*}
    This follows by assuming that $D_b T_{F, r}=0$ on the other part of the boundary of $\mathscr{A}$, which gives $r=0$ because of the linear independence of the components of $F$.
\end{remark}
Thus, we conclude that a.e. in $t \in (0, T), y \in Y$
\begin{align}
    \label{eq:argmin}
    \left( \bar{a}(y, t), \bar{b}(y, \cdot, t) \right) \in \mathrm{argmin} \; T_{\bar{F}(\cdot, t), \bar{r}(y, t)}
\end{align}
such that for $j = 1, \dots, N$
\begin{align}
    \label{eq:min_principle}
    \begin{cases}
        \partial_t \bar{f^{(j)}} = \sigma \left( \bar{a} - B_{\bar{b}} \bar{f^{(j)}} \right), &\quad 0 < t \le T \\
        \partial_t \bar{r^{(j)}} = B_{\bar{b}^*} \left( \sigma' \left( \bar{a} - B_{\bar{b}} \bar{f^{(j)}} \right) \bar{r^{(j)}} \right) &\quad 0 < t \le T \\
        \bar{f^{(j)}}(t=0) = f_I^{(j)} \\
        \bar{r^{(j)}}(y, T) = \int_U \left( \bar{P^{(j)}_{\mathrm{pre}}}(u) - P^{(j)}(u) \right) h' \left( \bar{Z^{(j)}}(u) \right) w(u, y) du,
    \end{cases}
\end{align}
where $\bar{Z^{(j)}}(u) := \int_Y w(u, y) \bar{f^{(j)}}(y, T) dy + \mu(u)$ and $\bar{P^{(j)}_\mathrm{pre}}(u) = h \left( \bar{Z^{(j)}}(u) \right)$. Note that the forward and backward problems are all coupled through the optimal control $\left( \bar{a}, \bar{b} \right)$ which depend on $\bar{f^{(1)}}, \dots, \bar{f^{(N)}}, \bar{r^{(1)}}, \dots, \bar{r^{(N)}}$.

Consider now the regression problem $h(\xi) = \xi$ for all $\xi \in \R$. Then we conclude
\begin{proposition} \label{proposition_7}
    Let $\left( \bar{a}, \bar{b} \right) \in \mathscr{C}$ be an optimal control and let $\{ f^{(j)}_I \}_{j=1, \dots, N}$ be weakly linearly independent. Then either
    \begin{enumerate}
        \item $\left( \bar{a}, \bar{b} \right) \in \mathring{\mathscr{C}}$ and $f^{(j)}(\cdot, T)$ $j= 1, \dots, N$ are least-square solutions of "$Wf^{(j)}(\cdot, T) = P^{(j)} - \mu$" \\ or
        \item $\left( \bar{a}, \bar{b} \right) \in \partial \mathscr{C}$.
        \item \label{condition:small_time} For $t \in (0, \epsilon)$ a.e. with $\epsilon$ sufficiently small and $y \in Y$ a.e. we have $\left( \bar{a}(y, t), \bar{b}(y, \cdot, t) \right) \in \partial \mathscr{A}$.
    \end{enumerate}
    \begin{proof}
        (\ref{condition:small_time}) follows from the fact that weak linear independence of $\{ f^{(j)}_I \}_{j=1, \dots, N}$ implies weak linear independence of $\{ f^{(1)}(\cdot, t), \dots, f^{(N)}(\cdot, t) \}$ for $0 \le t \le \epsilon$ with $\epsilon$ sufficiently small and we use Proposition \ref{proposition_6}
    \end{proof}
\end{proposition}

Note that Pontryagin's minimum principle states only a necessary condition which minimizers have to satisfy \cite{pontryagin2018mathematical}. However, in many instances the minimum principle can be used to significantly constrain or even determine the set of potential argmins, without providing sufficient conditions for the existence of minimizers. The example below the following remark will serve as an interesting illustration of this fact.

\begin{remark}
    The Hamiltonian \eqref{eq:hamiltonian} is constant in time at the optimal state \cite{evans2005introduction, little1996pontryagin}, i.e.,
    \begin{align*}
        H \left( \bar{F}, \bar{r}, \bar{a}, \bar{b} \right) = \mathrm{const} \quad \mathrm{on} \; [0, T].
    \end{align*}
    Note that this is non-trivial since $\bar{a}, \bar{b}$ depend on $t$ and are possibly not (everywhere) differentiable.
\end{remark}

The backward-forward coupling and in particular the coupling between different components in \eqref{eq:argmin}, \eqref{eq:min_principle} results in a highly nonlinear initial-terminal value problem. To illustrate this let us consider the mathematically interesting but practically irrelevant case of one set of training data $(f_I, P)$, i.e., the case $N=1$ and $A = [a_m, a_M] \times [b_m, b_M]$. We refer to \cite{liu2020selection}, where this case was already studied but we restate the result here for the reader's convenience. 
Let $\sigma' > 0$ on $\R$, then the minimization problem \eqref{eq:argmin} becomes
\begin{align*}
    \left( \bar{a}(y, t), \bar{b}(y, \cdot, t) \right) \in \mathop{\mathrm{argmin}}_{ (a, b) \in \mathscr{A}} \bigg[ \sigma \left( a - \int_Y b(z) \bar{f}(z, t) dz \right) \bar{r}(y, t) \bigg].
\end{align*}
We find, since $\sigma$ is strictly increasing, that a.e. in $(0, T)$ and a.e. in $\{ y \in Y : \bar{r}(y, t) \neq 0 \}$:
\begin{align*}
    \bar{a}(y, t) &= a_m \mathbbm{1}_{\{ \bar{r}(\cdot, t) > 0 \}}(y) + a_M \mathbbm{1}_{\{ \bar{r}(\cdot, t) < 0 \}}(y) \\
    \bar{b}(y, z, t) &= b_m \mathbbm{1}_{ \{ \bar{r}(\cdot, t) \otimes \bar{f}(\cdot, t) < 0 \} }(y, z) + b_M \mathbbm{1}_{ \{ \bar{r}(\cdot, t) \otimes \bar{f}(\cdot, t) > 0 \} }(y, z).
\end{align*}
Clearly, $\bar{f} = \bar{f}^{(1)}$, $\bar{r} = \bar{r}^{(1)}$ solve \eqref{eq:min_principle}, which is fully defined by the optimal bang-bang control $(\bar{a}, \bar{b})$ if and only if for a.e. $t \in (0, T)$ the $d$-dimensional Lebesgue measure of $\{ y \in Y : \bar{r}(y, t) = 0 \}$ vanishes. Note that the selection of $b$ on $\{ y \in Y : \bar{f}(y, t)=0 \}$ is of no significance to \eqref{eq:min_principle}.

\begin{section}{Dynamic Programming Principle and Hamilton-Jacobi-Bellman equation} \label{section:HJB}
    The two most widely used methods in control theory are the Pontryagin Maximum Principle and the Dynamic Programming Principle \cite{evans2005introduction}. While the former only provides a necessary condition for optimality the latter also gives (in a sense) a 'sufficient' condition albeit at the expense of much greater complexity. In this section, we extend the latter alternative approach to the deep learning residual network control problem presented in Section \ref{section:controllability}, building upon \cite{liu2020selection}.

    To begin, we define the value functional. Let $t \in [0, T]$, $s \in [t, T]$ and consider the forward problem for $f^{(j)} = f^{(j)}(y, s; t)$ with general initial data $v_j \in L^2(Y)$ imposed at $s=t$ for all $j=1, \dots, N$, i.e.,
    \begin{align}
        \label{eq:forward_v}
        \begin{cases}
            \partial_s f^{(j)} = \sigma \left( a - B_b f^{(j)} \right), \quad &y \in Y, s \in (t, T] \\
            f^{(j)}(y, s=t; t) = v_j(y) \quad &y \in Y,
        \end{cases}
    \end{align}
    with $(a, b) \in \mathscr{C}$ (as in Section \ref{section:max_principle}). Let the observed label functions $P^{(1)}, \dots, P^{(N)} \in L^2(U)$ be fixed and consider a given nonlinear, continuous cost functional $\mathcal{C} = \mathcal{C}(z, p) : L^2(U) \times L^2(U) \rightarrow \R$ such that $\mathcal{C} \ge 0$ on $L^2(Y) \times L^2(Y)$ and $\mathcal{C}$ = 0 if and only if $z = p$. 
    As in Section \ref{section:max_principle}, let $w \in L^2(U \times Y)$, $\mu \in L^2(U)$ be fixed and write the network output function
    \begin{align*}
        Z^{(j)}(u; t) = \int_Y w(u, y) f^{(j)}(y, T; t) dy + \mu(u).
    \end{align*}
    We define the value functional $\mathcal{V} = \mathcal{V} \left( v_1, \dots, v_N, t \right) : L^2(Y)^N \times [0, T] \rightarrow \R$ as
    \begin{align}
        \label{eq:value_functional}
        \mathcal{V} \left( v_1, \dots, v_N, t \right) := \inf_{(a, b) \in \mathscr{C}} \frac{1}{N} \sum_{j=1}^N \mathcal{C} \left( Z^{(j)}(\cdot, t), P^{(j)} \right).
    \end{align}
    Note that $f^{(j)}(y, T; T) = v_j(y)$ for all $y \in Y$ and for all $(a, b) \in \mathscr{C}$ implies $Z^{(j)}(u; T) = \int_Y w(u, y) v_j(y) dy + \mu(u)$ and
    \begin{align}
    \label{eq:terminal_condition_v}
        \mathcal{V} \left( v_1, \dots, v_N, T \right) = \frac{1}{N} \sum_{j=1}^N \mathcal{C} \left( \int_Y w(\cdot, y) v_j(y) dy + \mu, P^{(j)} \right) =: g\left( v_1, \dots, v_N \right),
    \end{align}
    with $g \in C(L^2(Y)^N; \R).$
    
    In the remainder of this section, we will rely on the following definitions of Lipschitz continuity for $\mathcal{C}$ and $\mathcal{V}$, which are provided here for clarity of the exposition.
    \begin{definition}
        We say that $\mathcal{C}$ is locally Lipschitz continuous with respect to its first argument in $L^2(U)$ if for $p \in L^2(U)$ and all $R > 0$ there exists $K = K(R, p)$ such that 
            \begin{align*}
                | \mathcal{C}(z_1, p) - \mathcal{C}(z_2, p) | \le K(R, p) \| z_1 - z_2 \|_{L^2(U)} \quad \mathrm{whenever} \; \| z_1 \|_{L^2(U)}, \| z_2 \|_{L^2(U)} \le R.
            \end{align*}
    \end{definition}

    \begin{definition}
        We say that the value functional $\mathcal{V}$ is locally Lipschitz continuous with respect to its full argument $\left( v_1, \dots, v_N, t \right)$ if for every $R > 0$ there exists $K = K(R)$ such that 
        \begin{align*}
            \bigg{|} \mathcal{V} \left( v^1_1, \dots, v_N^1, t_1 \right) - \mathcal{V} \left( v_1^2, \dots, v_N^2, t_2 \right) \bigg{|} \le K(R) \left( \sum_{j=1}^N \| v^1_{j} - v^2_{j} \|_{L^2(Y)} + | t_1 - t_2 | \right),
        \end{align*}
        whenever $\| v_j^1 \|_{L^2(Y)}, \| v_j^2 \|_{L^2(Y)} \le R$ for all $j=1, \dots, N$ and $0 \le t_1, t_2 \le T$.
    \end{definition}

    The definitions of uniform continuity and global Lipschitz continuity are immediate.

    \begin{proposition} \label{prop:lipschitz}
            \begin{enumerate}[label=(\roman*)]
            \item \label{i} If $\mathcal{C}$ is locally Lipschitz continuous with respect to its first argument in $L^2(U)$ then $\mathcal{V}$ is locally Lipschitz continuous with respect to its full argument in $L^2(Y)^N \times [0, T]$. 
            \item If $\mathcal{C}$ is globally Lipschitz continuous with respect to its first argument in $L^2(U)$ then $\mathcal{V}$ is globally Lipschitz continuous on $L^2(Y)^N$ uniformly for $t \in  [0,T]$ and locally Lipschitz continuous on bounded subsets of $L^2(Y)°^N \times [0, T]$.
            \item If $\mathcal{C}$ is uniformly Lipschitz continuous with respect to its first argument and if the activation function $\sigma$ is bounded on $\R$ then $\mathcal{V}$ is uniformly Lipschitz continuous.
            \item If $\mathcal{C}$ is uniformly continuous with respect to its first argument in $L^2(U)$ then $\mathcal{V}$ is uniformly continuous on $L^2(Y)^N$ uniformly for $t \in [0,T]$ and locally uniformly continuous on bounded subsets of $L^2(Y)^N \times [0, T]$.
        \end{enumerate}

        \begin{proof}
            We will prove \ref{i} and briefly comment on the other cases. Denote as in \eqref{eq:value_functional}
            \begin{align*}
                \mathcal{V} \left( v_1, \dots, v_N, t \right) = \inf_{(a, b) \in \mathscr{C}} G_{\left( v_1, \dots, v_N, t \right)} (a, b) := \inf_{(a, b) \in \mathscr{C}} \frac{1}{N} \sum_{j=1}^N \mathcal{C} \left( Z^{(j)}(\cdot, t), P^{(j)} \right).
            \end{align*}
            We start by analyzing the Lipschitz continuity of the map $\left( v_1, \dots, v_N, t \right) \rightarrow G_{\left( v_1, \dots, v_N, t \right)} (a, b)$. For all $j=1, \dots, N$, let $f^{(j)}_1, f^{(j)}_2$ satisfy \eqref{eq:forward_v} with initial conditions $v_j^1, v_j^2$, respectively, imposed at $s=t_1$ and $s=t_2$ respectively, and for a given pair $(a, b) \in \mathscr{C}$. 
            We estimate using \eqref{eq:estimate_for_f}, for $l=1,2$:
            \begin{align*}
                \| Z_l^{(j)}(\cdot \, ; t_l) \|_{L^2(U)} &\le \| w \|_{L^2(U\times Y)} \| f_l^{(j)} (\cdot, T; t_l) \|_{L^2(Y)} + \| \mu \|_{L^2(U)} \\
                &\le C_1 \| v^l_{j} \|_{L^2(Y)} + C_2,
            \end{align*}
            where $C_1, C_2$ are independent of $(a, b) \in \mathscr{C}$ and of $v_l^{(j)}$. Now, let $\| v^l_{j} \|_{L^2(Y)} \le R$ for $j=1, \dots, N$, $l=1,2$. Using the local Lipschitz continuity of $\mathcal{C}$ and denoting $\tilde{R}:= C_1 R + C_2$ we obtain
            \begin{align*}
                \bigg{|} \mathcal{C} \left(Z^{(j)}_1 (\cdot, t_1), P^{(j)} \right) - \mathcal{C} \left(Z^{(j)}_2 (\cdot, t_2), P^{(j)} \right) \bigg{|} &\le K(\tilde{R}, P^{(j)}) \| Z^{(j)}_1 (\cdot, t_1) - Z^{(j)}_2 (\cdot, t_2) \|_{L^2(U)} \\
                &\le K(\tilde{R}, P^{(j)}) \| w \|_{L^2(U \times Y)} \left( I_{t_2} + I_{t_1, t_2} \right),
            \end{align*}
            where
            \begin{align*}
                I_{t_2} &:= \| f_1^{(j)}(\cdot, T; t_2) - f_2^{(j)}(\cdot, T; t_2) \|_{L^2(Y)}, \\
                I_{t_1, t_2} &:= \| f_1^{(j)}(\cdot, T; t_1) - f_1^{(j)}(\cdot, T; t_2) \|_{L^2(Y)}.
            \end{align*}
            Using estimate \eqref{eq:estimate_1} and the uniform boundedness of $a, b$ we obtain the following bound
            \begin{align*}
                I_{t_2} \le K_1 \| v^1_{j} - v^2_{j} \|_{L^2(Y)},
            \end{align*}
            where $K_1$ is a constant independent of $v^1_{j}, v^2_{j}, a, b$. Similarly, employing \eqref{eq:uniform_time_estimate} and again the uniform boundedness of $a, b$ we get
            \begin{align*}
                I_{t_1, t_2} \le K_2(R) | t_1 - t_2 |,
            \end{align*}
            where $K_2$ is independent of $a, b$. Note that $K_2$ can be chosen independent of $R$ if $\sigma$ is bounded on $\R$. Thus we proved local Lipschitz continuity of the map $\left( v_1, \dots, v_N, t \right) \rightarrow G_{\left( v_1, \dots, v_N, t \right)} (a, b)$. Finally, to prove the Lipschitz continuity of $\mathcal{V}$ we recall that infima of $L$-Lipschitz maps are $L$-Lipschitz. The statements on uniform Lipschitz continuity and on uniform continuity follow analogously.
        \end{proof}
    \end{proposition}
    The proof of \ref{i} can easily be modified to show that the (assumed) continuity of $\mathcal{C}$ implies equicontinuity in the controls $(a,b) \in \mathscr{C}$ of the map $(v_1,..,v_N, t) \rightarrow G_{(v_1,..v_N,t)} (a,b)$. This proves continuity of $g \in L^2(Y)^N$.
    
    It is important to check the Lipschitz continuity assumption of the Proposition for the last-layer activation with the loss functions used in ML. For the regression task \eqref{eq:predicted_outcome} with $h(\xi) = \xi$ on $\R$ and the $L^2$-loss \eqref{eq:loss}, local Lipschitz continuity of $\mathcal{C}=\mathcal{C}(z,p)$ in $z \in L^2(U)$ is satisfied. Global Lipschitz continuity holds if the $L^2$-loss function is replaced by the $L^1$-loss, i.e., the mean absolute error.

    For the multi-label classification \eqref{eq:predicted_outcome}, \eqref{eq:loss} global Lipschitz continuity holds (since $h = h{(\xi})$ is bounded on $\R$ and it is globally Lipschitz continuous). In the case of the soft-max final layer activation \eqref{eq:soft_max} and the MSE or $L^1$-loss we have local Lipschitz continuity of $\mathcal{C} = \mathcal{C}(z, p)$ in $z \in L^2(U)$ if $\mu \in L^{\infty}(U, du)$, $w \in L^2(Y; L^{\infty}(U, du))$. The same holds for the cross entropy loss with soft-max activation in \eqref{eq:loss_log}.

    For the following, we recall that a Borel set $A$ in a separable Banach Space $X$ is said to be Gauss null if $\mu(A)=0$ for every non-degenerate Gaussian measure $\mu$ on $X$ (not supported on a proper closed hyperplane), see \cite{lindenstrauss2003frechet}. We refer to \cite{kuo2006gaussian, lunardi2015infinite} for the definition and construction of Gaussian measures on Banach spaces. 

    We now prove:
    \begin{theorem}
        Let $\mathcal{C} = \mathcal{C}(z, p)$ be locally Lipschitz continuous with respect to its first argument in $L^2(U)$. Then:
        \begin{enumerate}
            \item \label{differentiability_v} The value functional $\mathcal{V} : L^2(Y)^N \times [0, T] \rightarrow \R$ is Gateaux/Hadamard differentiable except on a Gauss null subset of $L^2(Y)^N \times [0, T]$.
            \item Let $\left( \tilde{v}_1, \dots, \tilde{v}_N, \tilde{t} \right)$ be a point of Gateaux/Hadamard differentiability of $\mathcal{V}$. Then, its Gateaux/Hadamard derivative $\left( D_{(v_1, \dots, v_N)} \mathcal{V}, D_t \mathcal{V} \right) \in L^2(Y)^N \times \R$ satisfies the functional Hamilton-Jacobi-Bellman (HJB) equation
            \begin{align}
                \label{eq:HJB}
                D_t \mathcal{V} + H_{\mathrm{HJB}} \left( v_1, \dots, v_N, D_{\left( v_1, \dots, v_N \right)} \mathcal{V} \right) = 0
            \end{align}
            at $\left( \tilde{v}_1, \dots, \tilde{v}_N, \tilde{t} \right)$ where the Hamiltonian $H_{\mathrm{HJB}}$ is given by
            \begin{align}
            \label{eq:H_HJB}
                H_{\mathrm{HJB}} \left( v_1, \dots, v_N, r_1, \dots, r_N \right) := \inf_{(a,b) \in \mathscr{A}_{\mathrm{HJB}}} \sum_{j=1}^N \int_Y \sigma \left( a(y) - B_{b(y, \cdot)} v_j \right) r_j(y) dy.
            \end{align}
        \end{enumerate}
    \end{theorem}
    \begin{remark}
        More explicitly, the Theorem says that if $\mathcal{C}$ is locally Lipschitz continuous with respect to its first argument in $L^2(U)$ then at each point $\left( \tilde{v}_1, \dots, \tilde{v}_N, \tilde{t} \right)$ outside a Gauss null subset of $L^2(Y)^N \times [0, T]$, the HJB equation \eqref{eq:HJB} holds. 
    \end{remark}
    \begin{proof}
        Since $\mathcal{C}$ is locally Lipschitz with respect to its first argument in $L^2(U)$, applying Proposition \ref{prop:lipschitz}, it follows that the value functional $\mathcal{V}$ is locally Lipschitz continuous with respect to its full argument. Consequently, $\mathcal{V}$ is Gateaux/Hadamard differentiable except on a Gauss null set due to an infinite-dimensional version of Rademacher's theorem, see \cite[Theorem 1.1]{lindenstrauss2003frechet} or \cite[Section 2, Theorem 1]{aronszajn1976}. This proves (\ref{differentiability_v}).

        Moreover, at each point where $\mathcal{V}$ is Gateaux/Hadamard differentiable, we can apply the chain rule, and \eqref{eq:HJB} follows directly by standard control theory arguments, see \cite[Theorem 5.1]{evans2005introduction}.
    \end{proof}
    
    The functional HJB equation \eqref{eq:HJB}, together with \eqref{eq:terminal_condition_v} constitute a terminal value problem posed on $L^2(Y)^N \times [0, T]$. Note that the Hamiltonian $H_{\mathrm{HJB}} = H_{\mathrm{HJB}} (v, r)$ is a concave functional of the 'gradient variable' $r = (r_1, \dots, r_N) \in L^2(Y)^N$ for every $v = (v_1, \dots, v_N) \in L^2(Y)^N$ since it is defined as the pointwise infimum of concave functionals. After the time reversal $\tau \rightarrow T-t$ the problem \eqref{eq:HJB}, \eqref{eq:terminal_condition_v} becomes an IVP with a Hamiltonian which is convex in the 'gradient variable'. Note that the Hamiltonian can be written as
    \begin{align*}
        H_{\mathrm{HJB}}(v, r) = \int_Y \min_{ \left( a, b \right) \in \mathscr{A} }  T_{v, r(y)} \left( a, b \right) dy,
    \end{align*}
    where $\mathscr{A}$, $T_{v, r(y)}$ are defined in Section \ref{section:max_principle}. The infimum is actually a minimum according to Proposition \ref{proposition_6} and a straightforward variant of Proposition \ref{prop:measurability}.
    
    We remark that the map $\mathcal{Q} : L^2(Y)^2 \rightarrow \R$ defined by $\mathcal{Q}(v, r) := \int_Y \sigma \left( a(y) - (B_{b(y, \cdot)} v)(y) \right) r(y) dy$ is locally Lipschitz continuous on $L^2(Y)^2$ uniformly in $(a, b) \in \mathscr{A}$. In fact we have, for $(a, b) \in \mathscr{A}$
    \begin{align}
    \label{eq:inequality_Q}
        | Q(v_1, r_1) - Q(v_2, r_2) | &\le \| \sigma \left( a - B_b v_1 \right) \|_{L^2(Y)} \| r_1 - r_2 \|_{L^2(Y)} + \| \sigma' \|_{L^{\infty}(\R)} \| b \|_{L^2(Y \times Y)} \| r_2 \|_{L^2(Y)} \| v_1 - v_2 \|_{L^2(Y)} \notag \\
        &\le K \left( \left( 1 + \| v_1 \|_{L^2(Y)} \right) \| r_1 - r_2 \|_{L^2(Y)} + \| r_2 \|_{L^2(Y)} \| v_1 - v_2 \|_{L^2(Y)} \right),
    \end{align}
    where $K$ is independent of $(v_1, r_1), (v_2, r_2)$. Furthermore, since $T_{v, r} (a, b) = \sum_{j=1}^N Q(v_j, r_j)$, as an infimum of a family of uniformly local Lipschitz continuous functionals $H_{\mathrm{HJB}}$ is obviously locally Lipschitz continuous on $L^2 (Y)^{2N}$.
    
    To give an example we return to the case $N=1$ and $A = [a_m, a_M] \times [b_m, b_M]$ discussed at the end of Section \ref{section:max_principle}. The Hamiltonian for the HJB can easily be computed and we find
    \begin{align*}
        H_{\mathrm{HJB}}(v, r) = &\sigma \left( a_m - b_m \int_{f(z)<0} f(z) dz - b_M \int_{f(z)>0} f(z) dz \right) \int_{r(y)>0}  r(y) dy \\
        &+ \sigma \left( a_M - b_m \int_{f(z)>0} f(z) dz - b_M \int_{f(z)<0} f(z) dz \right) \int_{r(y)<0}  r(y) dy,
    \end{align*}
    showing the high degree of nonlinearity in the HJB equation. Note that $\sigma' \ge 0$ suffices for this computation of the Hamiltonian, strict monotonicity of $\sigma$ is not required.

    It is well known in mathematical control theory that under various sets of assumptions the value functional is the unique viscosity solution of the terminal value problem for the HJB equation. To proceed in this direction we start with the definition of viscosity solution, following Definition 1.1 in \cite{crandall1986hamilton}.
    \begin{definition}
        $\mathcal{V} \in C \left( L^2(Y)^N \times [0, T] ; \R \right)$ is a viscosity solution of
        \begin{align*}
            \begin{cases}
                \partial_t \mathcal{V} + H_{\mathrm{HJB}} (v, D \mathcal{V}) = 0 \\
                v(t=T) = g
            \end{cases}
        \end{align*}
        if and only if:
        \begin{enumerate}[label=(\roman*)]
            \item $\mathcal{V}(v, t=T) = g(v)$ for all $v \in L^2(Y)^N$,
            \item whenever $\varphi \in C(L^2(Y)^N \times (0, T) ; \R)$, $v_0 \in L^2(Y)^N$, $t_0 \in (0, T)$, $\varphi$ is (Fréchet) differentiable at $(v_0, t_0)$ and $\mathcal{V}-\varphi$ has a local maximum at $(v_0, t_0)$ then
            \begin{align*}
                \partial_t \varphi (v_0, t_0) + H_{\mathrm{HJB}} \left( v_0, D \varphi(v_0, t_0) \right) \ge 0,
            \end{align*}
            and whenever $\varphi \in C(L^2(Y)^N \times (0, T) ; \R)$, $v_0 \in L^2(Y)^N$, $t_0 \in (0, T)$, $\varphi$ is (Fréchet) differentiable at $(v_0, t_0)$ and $\mathcal{V}-\varphi$ has a local minimum at $(v_0, t_0)$ then
            \begin{align*}
                \partial_t \varphi (v_0, t_0) + H_{\mathrm{HJB}} \left( v_0, D \varphi(v_0, t_0) \right) \le 0.
            \end{align*}
        \end{enumerate}
    \end{definition}
    Note that the signs in the definitions of viscosity sub- and super solutions of initial value problems are reversed here due to the fact that we are dealing with a terminal value problem. For simplicity's sake, we have dropped the subscript '$v$' to denote derivatives of functionals with respect to the vector-valued function $v$.

    We now denote by $UC (L^2(Y)^N; \R)$ the space of uniformly continuous real-valued functionals on $L^2(Y)^N$, by $BUC(L^2(Y)^N; \R)$ its subspace of bounded functionals and by $UC_s (L^2(Y)^N \times [0, T]; \R)$ the space of those functionals $F : L^2(Y)^N \times [0, T] \rightarrow \R$ which are uniformly continuous in their first argument $v \in L^2(Y)^N$ uniformly for $t \in [0, T]$ and uniformly continuous on bounded subsets of $L^2(Y)^N \times [0, T]$, i.e., there is a global modulus of continuity $m_0$ and a local one $m_1$ such that
    \begin{align*}
        | F(v_1, t) - F(v_2, t) | \le m_0 \left( \| v_1 - v_2 \|_{L^2(Y)^N} \right) + m_1 \left( \| v_2 \|_{L^2(Y)^N}, |t - s| \right), \quad \forall v_1, v_2 \in L^2(Y)^N, \forall t, s \in [0, T].
    \end{align*}
    $BUC_s (L^2(Y)^N \times [0, T]; \R)$ is defined in analogy.
    We now prove:
    \begin{theorem}
        The value functional $\mathcal{V}$ defined in \eqref{eq:value_functional} is a viscosity solution of the terminal value problem for the HJB-equation \eqref{eq:HJB}, \eqref{eq:terminal_condition_v} on $L^2(Y)^N \times [0, T]$. Moreover, if $\mathcal{C}$ is uniformly continuous with respect to its first argument in $L^2(Y)^N$ then:
        \begin{enumerate}
            \item \label{case1} if the activation function $\sigma$ is bounded on $\R$, $\mathcal{V}$ is the unique viscosity solution in $UC_s (L^2(Y)^N \times [0, T]; \R)$,
            \item \label{case2} if $\mathcal{C}$ is bounded on $L^2(Y)^N$ with respect to its first argument, then $\mathcal{V}$ is the unique viscosity solution in $BUC_s (L^2(Y)^N \times [0, T]; \R)$.
        \end{enumerate}
    \end{theorem}
    \begin{proof}
        The proof of Theorem 2 in Section 10.3 of \cite{evans10} can be used without modification to show in our infinite dimensional setting that the value functional $\mathcal{V}$ is a viscosity solution of the terminal-value problem for the HJB-equation. More needs to be done to prove uniqueness. 
        We shall now proceed to verify the assumption (H1)-(H4) of Theorem 1.1 in \cite{crandall1986hamilton}:
        \begin{enumerate}[label=(H\arabic*)]
            \item We have shown above that the Hamiltonian $H_{\mathrm{HJB}}$ is locally Lipschitz continuous in both arguments $v, r$.
            \item $H_{\mathrm{HJB}}$ is independent of $\mathcal{V}$ so this hypothesis does not apply.
            \item \label{H3} Let $\nu : L^2(Y)^N \rightarrow \R$ be non-negative, Fréchet differentiable on $L^2(Y)^N$, with bounded Fréchet derivative $D \nu(v) \in L^2(Y)^N$ and such that $\liminf_{|v| \rightarrow \infty} \frac{\nu(v)}{|v|} > 1$. Let $\lambda > 0$. We have from \eqref{eq:inequality_Q}
            \begin{align*}
                | H_{\mathrm{HJB}}(v, r) - H_{\mathrm{HJB}}(v, r + \lambda D \nu(v) ) | \le \lambda \sup_{v \in L^2(Y)^N} \| \sigma (v) \|_{L^2(Y)^N} \| D \nu(v) \|_{L^2(Y)^N}.
            \end{align*}
            If the activation function $\sigma$ is bounded, then \ref{H3} is satisfied, for instance, for $\nu(v) = \frac{1}{2} \sqrt{\| v \|^2_{L^2(Y)^N} + \frac{1}{2}}$. Restricting to bounded viscosity solution we can weaken the assumption on the $\liminf$ of $\nu(v)$ to $\nu(v) \rightarrow \infty$ as $|v| \rightarrow \infty$. Then we can invoke Theorem 5.1 in \cite{crandall1986hamilton} and choose $\nu(v) = \frac{1}{2} \ln \left( 1 + \| v \|^2_{L^2(Y)^N} \right)$. We compute $D \nu(v) = \frac{v}{1+\| v \|^2_{L^2(Y)^N}}$ and \eqref{eq:inequality_Q} gives
            \begin{align*}
                | H_{\mathrm{HJB}}(v, r) - H_{\mathrm{HJB}}(v, r + \lambda D \nu(v) ) | \le K_1 \lambda
            \end{align*}
            where $K_1$ is independent of $v, r, \lambda$. Thus \ref{H3} is satisfied in both cases \ref{case1} and \ref{case2}.
            \item \label{H4} Set $d(v_1, v_2) = \| v_1 - v_2 \|_{L^2(Y)^N}$ and estimate for $\lambda > 0$, using \eqref{eq:inequality_Q}
            \begin{align*}
                &\big{|} H_{\mathrm{HJB}} \left( v_1, -\lambda D_{v_1} d(v_1, v_2) \right) - H_{\mathrm{HJB}} \left( v_2, -\lambda D_{v_2} d(v_1, v_2) \right) \big{|} \\
                &= \bigg{|} H_{\mathrm{HJB}} \left( v_1, -\lambda \frac{(v_1 - v_2)}{\| v_1 - v_2 \|_{L^2(Y)^N}} \right) - H_{\mathrm{HJB}} \left( v_2, -\lambda \frac{(v_1 - v_2)}{\| v_1 - v_2 \|_{L^2(Y)^N}} \right) \bigg{|} \le K_2 \lambda \| v_1 - v_2 \|_{L^2(y)^N}. 
            \end{align*}
            This verifies \ref{H4} since $K_2$ is independent of $\lambda, v_1, v_2$.
        \end{enumerate}
        \end{proof}

        To construct a feedback control we rewrite through \eqref{eq:hamiltonian}, \eqref{eq:H_HJB}
        \begin{align*}
            H_{\mathrm{HJB}}(v, r) &= \min_{(a, b) \in \mathscr{A}_{\mathrm{HJB}}} \int_Y \sum_{j=1}^N \sigma \left( a(y) - B_{b(y, \cdot)} v_j \right) r_j dy \\
            &= \min_{(a, b) \in \mathscr{A}_{\mathrm{HJB}}} H(v, r, a, b) \quad \mathrm{for} \; (v, r) \in L^2(Y)^{2N}.
        \end{align*}
        Let $(\tilde{a}, \tilde{b} ) : L^2(Y)^{2N} \rightarrow \mathscr{A}_{\mathrm{HJB}}$, where
        \begin{align*}
            (\tilde{a}(v, r), \tilde{b}(v, r) ) \in \mathrm{argmin}_{(a, b) \in \mathscr{A}_{\mathrm{HJB}}} H(v, r, a, b)
        \end{align*}
        and $(\tilde{a}(v, r), \tilde{b}(v, r) )$ is a Borel-measurable selection of the minimizer (see Proposition \ref{prop:measurability}). Then the HJB-equation reads
        \begin{align*}
            \begin{cases}
                \partial_t \mathcal{V}(v, t) + H\left(v, D_v \mathcal{V}(v, t), \tilde{a}(v, D_v \mathcal{V}(v, t)), \tilde{b}(v, D_v \mathcal{V}(v, t))\right) = 0 \\
                \mathcal{V}(v, t= T) = g(v).
            \end{cases}
        \end{align*}
        We define $\Sigma =\Sigma(v, t)(y) = \left( \Sigma_1, \dots, \Sigma_N \right) \in \R^N$ by $\Sigma_j(v, t)(y) := \sigma \left( \tilde{a}(v, D_v \mathcal{V}(v, t))(y) - \left( B_{\tilde{b}(v, D_v \mathcal{V}(v, t))} v_j \right) (y) \right)$ and rewrite the HJB-equation as
        \begin{align*}
            \partial_t \mathcal{V}(v, t) + \left( \Sigma(v, t), D_v \mathcal{V}(v, t) \right)_{L^2(Y)^N} = 0.
        \end{align*}
        Now, let $\tilde{F} = \left( \tilde{f}^{(1)}, \dots, \tilde{f}^{(N)} \right)^{\mathrm{tr}}$ and solve the forward problem for $\tilde{f}^{(j)}$, $j=1, \dots, N$:
        \begin{align*}
            \begin{cases}
                \partial_s \tilde{f}^{(j)}(y, s; t) = \sigma \left( \tilde{a}\left(\tilde{F}(\cdot, s; t), D_v \mathcal{V}(\tilde{F}(\cdot, s; t), t)\right)(y) - \left( B_{\tilde{b}\left(\tilde{F}(\cdot, s; t), D_v \mathcal{V}(\tilde{F}(\cdot, s; t), t)\right)} \tilde{f}^{(j)}(\cdot, s; t) \right) (y, s) \right) \\
                \tilde{f}^{(j)}(y, s=t; t) = v_j
            \end{cases}
        \end{align*}
        assuming that $\tilde{a}, \tilde{b}, \mathcal{V}$ are sufficiently smooth. Note that this IVP is similar to the one in \eqref{eq:forward_v}, with the key difference being that here the controls $(\tilde{a}, \tilde{b})$ depend nonlinearly on all the solutions of the forward problem $\tilde{f}^{(j)}$ and on the derivative of the value functional $D_v \mathcal{V}$, which introduces a significant degree of nonlinearity into the system. The system can be reformulated as
        \begin{align}
        \label{eq:system_Ftilde}
        \begin{cases}
            \partial_s \tilde{F}(y, s; t) = \Sigma \left( \tilde{F}(y, s; t), s \right) \\
            \tilde{F}(y, s=t; t) = v.
        \end{cases}
        \end{align}
        Formally we compute
        \begin{align*}
            \frac{d}{ds} \mathcal{V} \left( \tilde{F}(\cdot, s; t), s \right) = \partial_s \mathcal{V} \left( \tilde{F}(\cdot, s; t), s \right) + \left( \Sigma \left( \tilde{F}(\cdot, s; t), s \right), D_v \mathcal{V} \left( \tilde{F}(\cdot, s; t), s \right) \right)_{L^2(Y)^N} = 0.
        \end{align*}
        We define $J_{v, t}(a, b) := g \left( F_{a, b}(\cdot, T; t \right))$ such that $\mathcal{V}(v, t) = \inf_{(a, b) \in \mathscr{A}_{\mathrm{HJB}}} J_{v, t}(a, b)$ and the feedback controls
        \begin{align*}
            a^* = \tilde{a} \left( \tilde{F}(\cdot, s; t), D_v \mathcal{V}(\tilde{F}(\cdot, s; t), t) \right)(y), \quad b^* = \tilde{b} \left( \tilde{F}(\cdot, s; t), D_v \mathcal{V}(\tilde{F}(\cdot, s; t), t) \right)(y, z).
        \end{align*}
        Using the formal calculation from above, with $F_{a^*,  b^*} = \tilde{F}$ we get
        \begin{align*}
            J_{v, t} (a^*, b^*) &= g(\tilde{F}(\cdot, T; t) ) \\
            &= g(\tilde{F}(\cdot, T; t) ) - \int_t^T \frac{d}{ds} \mathcal{V} (\tilde{F}, s) ds \\
            &= g(\tilde{F}(\cdot, T; t) ) - \mathcal{V} (\tilde{F}(\cdot, T; t), T) + \mathcal{V} (\tilde{F}(\cdot, t; t), t ) \\
            &= g(\tilde{F}(\cdot, T; t) ) - g(\tilde{F}(\cdot, T; t) ) + \mathcal{V}(v, t) \\
            &= \mathcal{V} (v,  t) \\
            &= \inf_{(a, b) \in \mathscr{A}_{\mathrm{HJB}}} J_{v, t} (a,  b)
        \end{align*}
        Therefore, if the formal arguments can be made rigorous, we conclude that $(a^*, b^*)$ is an optimal (feedback) control for 
        \begin{align}
            \label{eq:final_system}
            \begin{cases}
                \partial_s F = \sigma \left( a e - B_b F \right), \quad t \le s \le T \\
                F(y, s=t; t) = v
            \end{cases}
        \end{align}
        minimizing the loss functional
        \begin{align}
            \label{eq:loss_feedback}
            J_{v, t}(a, b) = g \left( F(\cdot, T; t) \right).
        \end{align}
        We sum up the above discussion in the following Proposition.
        \begin{proposition}
            Assume 
            \begin{enumerate}[label=(\roman*)]
                \item $(\tilde{a}, \tilde{b}) : L^2(Y)^{2N} \rightarrow \mathscr{A}_{\mathrm{HJB}}$, where $(\tilde{a}(v, r), \tilde{b}(v, r) ) \in \mathrm{argmin}_{(a, b) \in \mathscr{A}_{\mathrm{HJB}}} H(v, r, a, b)$ for all $(v, r) \in L^2(Y)^{2N}$, is a Borel-measurable selection,
                \item the value function $\mathcal{V}$ is locally Lipschitz on $L^2(Y)^N \times [0, T]$,
                \item \eqref{eq:system_Ftilde} has a solution $\tilde{F} \in L^1 ( (0, t); L^2(Y)^N )$ for every $t \in (0, T)$,
                \item $\mathcal{V}$ is Gateaux differentiable on the arc $\{ (\tilde{F}(v, s; t), s) : v \in L^2(Y)^N, s \in [t, T] \}$ for all $t \in [0, T]$.
            \end{enumerate}
            Then $(a^*, b^*)$ is an optimal feedback control of \eqref{eq:final_system}, \eqref{eq:loss_feedback}.    
        \end{proposition}

        Going back to our original DNN problem, we need to set $v_j (y)  = f^{(j)}_I (y)$ and $t=0$.  Clearly, it seems like an overkill to solve the high dimensional HJB-equation (even before the continuum limit of Section \ref{section:forward_back} it 'lives' on $\R^{NM} \times [0,T]$ ) in order to find optimal feedback controls, but we remark that there are also advantages to the HJB approach. First of all, once the value functional $\mathcal{V}$ is known, it is easy to change the initial training data set, new optimal feedback controls are easily computed from \eqref{eq:system_Ftilde}. Moreover we remark that the numerical solution of the high dimensional HJB equation is the subject of intense scrutiny and various fast algorithms have been established \cite{nakamura2021adaptive, nusken2021solving}.
\end{section}

\bibliographystyle{plain}
\bibliography{bibliography}

\end{document}